\documentclass{amsart}

\newtheorem{theorem}{Theorem}[section]
\newtheorem{thm}[theorem]{Theorem}
\newtheorem{lemma}[theorem]{Lemma}
\newtheorem{prop}[theorem]{Proposition}
\newtheorem{cor}[theorem]{Corollary}
\theoremstyle{definition}
\newtheorem{defn}[theorem]{Definition}

\theoremstyle{remark}

\numberwithin{equation}{section}

\usepackage{amscd,amsthm,amssymb,amsfonts,amsmath,euscript,mathrsfs}

\usepackage{hyperref}

\DeclareMathOperator{\fchar}{char}
\DeclareMathOperator{\Mass}{Mass}
\DeclareMathOperator{\img}{img}
\DeclareMathOperator{\Emb}{Emb}
\newcommand{\absv}[1]{\lvert #1\rvert}
\newcommand{\Lsymb}[2]{\genfrac{(}{)}{}{}{#1}{#2}}  
\newcommand{\qalg}[3]{\left(\frac{#1, #2}{#3}\right)}

 \newcommand{\nc}{\newcommand}

\def\makeop#1{\expandafter\def\csname#1\endcsname
  {\mathop{\rm #1}\nolimits}\ignorespaces}
\makeop{Hom}   \makeop{End}   \makeop{Aut}   \makeop{Isom}  \makeop{Pic} 
\makeop{Gal}   \makeop{ord}   \makeop{Char}  \makeop{Div}   \makeop{Lie} 
\makeop{PGL}   \makeop{Corr}  \makeop{PSL}   \makeop{sgn}   \makeop{Spf}
\makeop{Spec}  \makeop{Tr}    \makeop{Nr}    \makeop{Fr}    \makeop{disc}
\makeop{Proj}  \makeop{supp}  \makeop{ker}   \makeop{im}    \makeop{dom}
\makeop{coker} \makeop{Stab}  \makeop{SO}    \makeop{SL}    \makeop{SL}
\makeop{Cl}    \makeop{cond}  \makeop{Br}    \makeop{inv}   \makeop{rank}
\makeop{id}    \makeop{Fil}   \makeop{Frac}  \makeop{GL}    \makeop{SU}
\makeop{Nrd}   \makeop{Sp}    \makeop{Tr}    \makeop{Trd}   \makeop{diag}
\makeop{Res}   \makeop{ind}   \makeop{depth} \makeop{Tr}    \makeop{st}
\makeop{Ad}    \makeop{Int}   \makeop{tr}    \makeop{Sym}   \makeop{can}
\makeop{length}\makeop{SO}    \makeop{torsion} \makeop{GSp} \makeop{Ker}
\makeop{Adm}
\def\makebb#1{\expandafter\def
  \csname bb#1\endcsname{{\mathbb{#1}}}\ignorespaces}
\def\makebf#1{\expandafter\def\csname bf#1\endcsname{{\bf
      #1}}\ignorespaces} 
\def\makegr#1{\expandafter\def
  \csname gr#1\endcsname{{\mathfrak{#1}}}\ignorespaces}
\def\makescr#1{\expandafter\def
  \csname scr#1\endcsname{{\EuScript{#1}}}\ignorespaces}
\def\makecal#1{\expandafter\def\csname cal#1\endcsname{{\mathcal
      #1}}\ignorespaces} 

\def\doLetters#1{#1A #1B #1C #1D #1E #1F #1G #1H #1I #1J #1K #1L #1M
                 #1N #1O #1P #1Q #1R #1S #1T #1U #1V #1W #1X #1Y #1Z}
\def\doletters#1{#1a #1b #1c #1d #1e #1f #1g #1h #1i #1j #1k #1l #1m
                 #1n #1o #1p #1q #1r #1s #1t #1u #1v #1w #1x #1y #1z}
\doLetters\makebb   \doLetters\makecal  \doLetters\makebf
\doLetters\makescr 
\doletters\makebf   \doLetters\makegr   \doletters\makegr
    \def\setminus{\smallsetminus}
\def\Gm{{{\bbG}_{\rm m}}}   

\normalsize

\makeop{Bl}

\def\Fpbar{\overline{\bbF}_p}
\def\Fp{{\bbF}_p}
\def\Fq{{\bbF}_q}

\def\Qp{{\bbQ}_p}

\def\Zp{{\bbZ}_p}

\newcommand{\Z}{\mathbb Z}
\newcommand{\Q}{\mathbb Q}
\newcommand{\R}{\mathbb R}

\newcommand{\D}{\mathbf D}    
\newcommand{\A}{\mathbb A}    
\renewcommand{\O}{\mathcal O} 
\newcommand{\F}{\mathbb F}

\newcommand{\npr}{\noindent }
\newcommand{\pr}{\indent }

\newcommand{\oneone}{\mbox{$\longleftrightarrow$}}  
\newcommand{\isoto}{\stackrel{\sim}{\longrightarrow}}
\nc{\embed}{\hookrightarrow}

\newcommand{\ch}{characteristic }
\newcommand{\ac}{algebraically closed }
\newcommand{\dieu}{Dieudonn\'{e} }

\nc{\ol}{\overline}
\nc{\wt}{\widetilde}
\nc{\opp}{\mathrm{opp}}
\def\ul{\underline}

\makeop{Ram}
\makeop{Rep}
\makeop{Mat}

\def\wh{\widehat}

\begin{document}
\numberwithin{equation}{section}


\title[Counting abelian varieties]
{On counting certain abelian varieties  
over finite fields}
\author{Jiangwei Xue}

\address{(Xue) Collaborative Innovation Center of Mathematics, School of
  Mathematics and Statistics, Wuhan University, Luojiashan, 430072,
  Wuhan, Hubei, P.R. China}   
\email{xue\_j@whu.edu.cn}

\author{Chia-Fu Yu}

\address{
(Yu) Institute of Mathematics, Academia Sinica and NCTS \\
6th Floor, Astronomy Mathematics Building \\
No. 1, Roosevelt Rd. Sec. 4 \\ 
Taipei, Taiwan, 10617} 
\email{chiafu@math.sinica.edu.tw}


\date{\today}

\begin{abstract}
This paper contains two parts toward studying abelian varieties
from the classification point of view.  
In a series of papers \cite{xue-yang-yu:num_inv,xue-yang-yu:ECNF,
  xue-yang-yu:sp_as,xue-yang-yu:sp_as2}, the current authors and
T.-C. Yang obtain explicit formulas for the numbers of 
superspecial abelian surfaces over
finite fields. 
In this paper, we give an explicit formula for the size of the 
isogeny class of
simple abelian surfaces with real Weil number $\sqrt{q}$. 
This establishes a key step that one may extend our previous 
explicit calculations 
of superspecial abelian surfaces to those of 
 supersingular abelian surfaces.
The second part is to introduce the notion of genera and
idealcomplexes of abelian varieties with
additional structures in a general setting. 
The purpose is to generalize the results of \cite{yu:smf} 
on abelian varieties with additional structures to similitude 
classes, which establishes more results on the connection between 
geometrically defined and
arithmetically defined masses for further investigation.  

    
\end{abstract} 

\maketitle

\section{Introduction}
\label{sec:intro}

Throughout this paper, $p$ denotes a prime number 
and $q$ is a power of $p$. 
Recall that an abelian variety over a field $k$ of characteristic $p$ 
is said to be {\it supersingular} if it is
isogenous to a product of supersingular elliptic curves over an
algebraic closure 
$\bar k$ of $k$; it is said to be {\it superspecial} if it is
isomorphic to a product of supersingular elliptic curves over $\bar
k$. It is known that any supersingular abelian variety is isogenous to a
superspecial abelian variety. Thus, studying superspecial abelian
varieties is a vital step for studying supersingular abelian varieties.   

By the work of Deuring \cite{deuring:1941} and 
the new input of Waterhouse \cite{waterhouse:thesis} 
using the elegant theory of Honda and Tate \cite{tate:ht}, 
we have an explicit formula for the size of each isogeny class of
elliptic curves over a finite field $\Fq$ (also see \cite{schoof:1987}). 
The main part of these formulas are discussing
various cases of supersingular elliptic curves. 
It is desire to have similar results for abelian
surfaces over $\Fq$. Abelian surfaces over $\Fq$ are divided into
ordinary, almost ordinary and supersingular ones. As in the case of 
elliptic curves, the classifications 
of ordinary and almost ordinary abelian surfaces are simpler; and the
simple classes have
been studied by Waterhouse \cite{waterhouse:thesis}.  
  
In a series of papers \cite{xue-yang-yu:num_inv,xue-yang-yu:ECNF,
  xue-yang-yu:sp_as,xue-yang-yu:sp_as2}, the current authors and
T.-C. Yang obtain an explicit formula for the number of 
$\Fq$-isomorphism classes of superspecial abelian surfaces in each
  isogeny class over $\Fq$. Our next step is to compute explicitly that
  for each isogeny of supersingular abelian surfaces. 
  These are of course special cases of the general question
  of how to compute explicitly the size of an $\Fq$-isogeny class.

Suppose that $q=p^a$ be a power of a prime number $p$.
For any Weil $q$-number $\pi$, let $\calA_\pi$ be the set of 
$\Fq$-isomorphism classes of
abelian varieties in the $\Fq$-isogeny class associated to $\pi$ 
by the Honda-Tate theorem. Let $A_\pi$ be a member in
$\calA_\pi$, and denote by $E_\pi:=\End^0(A_\pi)
=\End_{\Fq}(A_\pi)\otimes \Q$ 
its endomorphism algebra over $\Fq$,
which depends only on the Weil number $\pi$. In Section 2, we
give a description for $\calA_\pi$ in terms of double coset spaces;
see Theorem~\ref{1}. This can be viewed as a special case of the
simple mass formula \cite[Theorem 2.2]{yu:smf}, and also 
a generalization
of some main results of Waterhouse \cite{waterhouse:thesis}. 
The proof we give
here is more conceptual than that in \cite{yu:smf} as for finite ground
fields we have convenient tools of Tate modules and \dieu modules. 
This method already appears in  \cite{kottwitz:michigan1990} where
Kottwitz expresses the size of each isogeny class in $\calA_g(\Fq)$
in terms of orbital integrals.  
The main difference is that the description given here is in 
terms of a sum of class numbers, for which one may further 
compute their values using computer for each specific case or for some
special cases.    
When the center $F=\Q(\pi)$ is a CM field, we show that
$|\calA_\pi|$ is the sum of certain explicit ray class numbers of $F$
(see Proposition~\ref{2}). For the other case where $F$ is totally real,
we discuss the classification of ``genera'' in details, and one can
apply the generalized Eichler trace formula in \cite{xue-yang-yu:ECNF}
to compute the class numbers of these genera. 
For the reader's convenience, we describe the extended trace formula 
in Section 3.

Our first main result is Theorem~\ref{4.4}, which  gives an 
explicit formula for $|\calA_\pi|$, where $\pi=\pm \sqrt{q}$
with odd  exponent $a$.
Note that the case where the exponent $a$ is even is a classical
result of Deuring and Eichler. So Theorem~\ref{4.4} completes 
the explicit calculation of $|\calA_\pi|$ when
$F$ is totally real. 
The main tool  
is the extended version of Eichler's trace formula we just mention.
We use this to calculate 
the number of \emph{superspecial} members in $\calA_\pi$ 
first. For
computing the non-superspecial members we use the Morel-Bailly
family \cite[Section 1]{katsura-oort:surface}, 
which may be also viewed as a special case of minimal isogenies 
introduced in\footnote{We caution the reader
  that the
  proof of this lemma is incorrect because $a(X_i)>a(X_{i-1})$ may not
  hold even when $a(X_{i-1})<g$. See \cite{yu:note_ss} for
  the existence of the minimal isogeny for an arbitrary (not only
  supersingular) abelian variety.} \cite[Lemma 1.8]{li-oort}.    
In the course of our computation,   
the Drinfeld period domain of rank two 
over $\Fq$ shows up and plays an interesting role.
The method by minimal isogenies paves a way to reduce the 
computation of non-superspecial
part to that of superspecial ones, and we have already established 
explicit formulas for the latter. 
Some of our arguments in the proof applies to other
isogeny classes of supersingular abelian surfaces as well. 
 
The second part of this paper is to discuss analogies between abelian
varieties and lattices in a general framework. Based on this idea of
analogy with lattices, we introduce the notion of genera and
idealcomplexes of abelian varieties with additional structures in a
general setting.   
As a result, we extend some results of \cite{yu:smf} to similitude
classes, and establish more connections between algebraically defined
and geometrically defined masses.  
Though our general result is still rather
abstract, more explicit results may be established under this
framework. As an example, the second author gives an
explicit formula for the number of the superspecial locus in the good
reduction modulo $p$ of a general type C Shimura varieties; 
see \cite{yu:mass_c}. At the end of this paper we give one example as
the application of the main result
(Theorem~\ref{smf.7}).  
  


\section{Isomorphism classes in a simple isogeny class}\label{sec:isom}

In this section we discuss certain structures that may be used in 
the problem of (more) explicit computation of
isomorphism classes in an isogeny class of abelian varieties over
finite fields. As the problem of determining explicitly the size of
the number of isomorphism classes is 
quite difficult at this moment, 
we first consider simple isogeny classes for simplicity.
Theorem~\ref{1} may be viewed as a common generalization of main
results of Waterhouse (see the first part of Theorem 5.1 and Theorems 6.1
and 7.2 of \cite{waterhouse:thesis}). 
The method already appears in \cite{kottwitz:michigan1990} where
Kottwitz expresses the size of each isogeny class in $\calA_g(\Fq)$
in terms of orbital integrals.  
The main difference is that the description given here is in 
terms of a sum of class numbers, for which one may further 
compute the values using computer for each specific case. 
Based on the similar description, 
Lipnowski and Tsimerman \cite[Section 3]{Lip-Tsi} further 
give a nice bound for the
size of any isogeny class over $\Fp$ of arbitrary dimension $g$. 
  
\subsection{}
Let $k$ be a finite field of cardinality $q$, where $q=p^a$ is a power
of a prime number $p$. Let $\pi$ be a Weil $q$-number. By the
Honda-Tate theory \cite{honda,tate:ht}, 
there is a $k$-simple abelian variety $A_\pi$ over $k$,
uniquely determined up to $k$-isogeny, so that the Frobenius endomorphism
of $A_\pi$ over $k$ is 
conjugate to $\pi$. Let
$F:=\Q(\pi)$ and $O_F$ be the ring of integers. 
The field $F$ is either a CM field or a totally real field, 
because it is stable
under a positive involution (Rosati involution). 
Let $\calA_{\pi}$ be the set of $k$-isomorphism classes
of abelian varieties in the $k$-isogeny class of $A_\pi$. 
In this section we
present a method toward computing the cardinality of  $\calA_\pi$. 

Let $E_\pi$ be the endomorphism algebra of $A_\pi$ over $k$. It is a
central division algebra over $F$. The local invariants of $E_\pi$ 
are given by \cite{tate:eav}:
\[ \inv_v(E_\pi)=
\begin{cases}
  \frac{1}{2} & \text{if } v \text{ is real,} \\
   \frac{v(\pi)}{v(q)}[F_v:\Q_p] & \text{if }v|p, \\
   0 & \text{otherwise.}   
\end{cases} \]
Let $[E_\pi:F]=e^2$. Let $P(t)$ be the characteristic polynomial of
$\pi$ on $A_\pi$. By definition
this is the characteristic polynomial of the Frobenius endomorphism 
$\pi$ on
the Tate module $T_\ell(A_\pi)$ for any prime $\ell\neq p$, and
one has $P(t)\in \Z[t]$ . Let $m(t)$
be the minimal polynomial of $\pi$ on $A_\pi$; this is the same as 
that of $\pi$ in $E_\pi$ over $\Q$ and hence it is a irreducible
polynomial in $\Z[t]$. 
We have $P(t)=m(t)^e$. Let $R:=\Z[\pi]=\Z[t]/(m(t))\subset
F$, and for any prime $\ell$ (including $p$), write 
$R_\ell:=R\otimes_{\Z} \Z_\ell$.   
 
For any prime $\ell\neq p$, let $T_\ell(A_\pi)$ denote the Tate module
of $A_\pi$ and put $V_\ell(A_\pi):=T_\ell(A_\pi)\otimes_{\Z_\ell} \Q_\ell$. 
The Galois invariant $\Z_\ell$-lattices in $V_\ell(A_\pi)$ are 
nothing but $R_\ell$-lattices. Note that $V_\ell(A_\pi)$ is a free
$F_\ell$-module of rank $e$, where $F_\ell:=F\otimes_\Q \Q_\ell$. 
Define 
\[  \wt \grX_{\pi,\ell}:=\left  \{ \text{$R_\ell$-lattices in
    $V_\ell(A_\pi)$  
} \right \}, \quad \text{and}\quad \grX_{\pi,\ell}:=\wt
\grX_{\pi,\ell}/\simeq, \] 
where for any two $R_\ell$-lattices $L_1$ and $L_2$ in 
$\wt \grX_{\pi,\ell}$ we write $L_1\simeq L_2$ if they are isomorphic
as $R_\ell$-modules.  
Since the order $R_\ell$ is maximal for almost all primes $\ell$, the set
$\grX_{\pi,\ell}$ is a singleton for almost all primes $\ell$. 
 
At the prime $p$, let $M(A_\pi)$ denote the (covariant) \dieu module
of $A_\pi$ and let $N(A_\pi):=M(A_\pi)\otimes_{\Z_p} \Q_p$ be the
associated ($F$-)isocrystal. Define 
\[\wt \grX_{\pi,p}:=\left \{\text{Dieudonn\'e submodules in
    $N(A_\pi)$ of full rank} \right \}, \quad \text{and}\quad
\grX_{\pi,p}:=\wt \grX_{\pi,p} /\simeq, \]
where for any two \dieu submodules  $M_1$ and $M_2$ in 
$\wt \grX_{\pi,p}$ we write $M_1\simeq M_2$ if they are isomorphic
as \dieu modules.
Taking the product over all primes, we obtain a finite set
$ \grX_{\pi}:=\grX_{\pi,p}\times \prod_{\ell\neq p} \grX_{\pi,\ell}. $
The association to each abelian variety $A$ over $k$ its \dieu
module and all Tate modules defines a map $\Phi: \calA_\pi \to
\grX_{\pi}$.  

Recall that any quasi-isogeny $\varphi:A_1\to A_2$ of abelian
varieties over $k$ is an element 
$\varphi\in \Hom_k(A_1,A_2)\otimes_{\Z}\Q$ such that $N\varphi$ is an
isogeny for some integer $N$.  
Let us put 
\begin{gather*}
\wt \grX_{\pi}:=\wt \grX_{\pi,p}\times {\prod_{\ell\neq p}}' 
\, \wt \grX_{\pi,\ell}=\{ (M,(T_\ell)_{\ell\neq p})\,\mid \,
T_\ell=T_\ell(A_\pi), \forall\,' \ell\}, \quad \text{and} \\
\wt \calA_{\pi}:=\left \{ \text{quasi-isogenies $\varphi:A\to
    A_\pi$ over $k$} \right \}. 
\end{gather*}
Two quasi-isogenies $\varphi_1:A_1\to A_\pi$ and $\varphi_2:A_2\to
A_\pi$ over $k$ are regarded as the same element in $\wt
\calA_{\pi}$ if there exists a $k$-isomorphism $\alpha: 
A_1\to A_2$ such that $\varphi_2 \circ \alpha=\varphi_1$. 
Define the map $\wt \Phi: \wt \calA_{\pi}\to \wt \grX_\pi$ by 
\[ \wt \Phi(\varphi:A\to A_\pi):=\left (\varphi_*(M(A)),
\varphi_*(T_\ell(A))_{\ell\neq p}\, \right ), \]
which is a bijection due to a theorem of Tate \cite{tate:eav}. 
Then we have a natural commutative diagram:
\def\pr{{\rm pr}}
\begin{equation}
  \label{eq:diagram}
\begin{CD}
  \wt \calA_{\pi}@>{\wt \Phi}>{\sim}> \wt \grX_\pi \\
   @V{\pr_\calA}VV @VV{\pr_\grX}V \\
   \calA_{\pi} @>{\Phi}>>  \grX_\pi,
\end{CD}  
\end{equation}
where the vertical maps are natural surjective maps. It follows that
the map $\Phi:\calA_\pi\to \grX_{\pi}$ is also surjective. We consider the
fibers of this map.  

Let $G_\pi$ be the algebraic group over $\Q$ associated to the 
multiplicative group $E_{\pi}^\times$. By Tate's homomorphism theorem
on abelian varieties, one has canonical isomorphisms
\begin{equation}
  \label{eq:8}
 G_\pi(\Q_\ell)=\Aut_{F_\ell}(V_\ell(A_\pi)), \quad \text{and} \quad
G_{\pi}(\Q_p)=\Aut_{\rm DM}(N(A_\pi)). 
\end{equation}
Let $\A_f$ be the finite adele ring of $\Q$, and
$ X=([M], [T_\ell]_{\ell\neq p})$ an element of $\grX_{\pi}$.  Define
a compact subgroup $U_X\subset G_\pi(\A_f)$ by
\[ U_X:=\Aut_{\rm DM} (M) \times \prod_{\ell\neq p}
\Aut_{R_\ell}(T_\ell); \]
this is uniquely determined by $X$ up to conjugation by
$G_\pi(\A_f)$.  
The group
$G_\pi(\A_f)$ admits 
a natural left action on the set $\wt \grX_\pi$ of \dieu
and Tate modules. It also acts on $\wt A_{\pi}$ from the left through
the isomorphisms (\ref{eq:8}) (cf. Lemma~\ref{smf.2})
so that the map $\wt \Phi$ is $G_\pi(\A_f)$-equivariant.   
Denote by $\calA_{\pi,X}$ and $\wt \calA_{\pi,X}$ the fibers over $X$ in
$\calA_\pi$ and $\wt \calA_{\pi}$, respectively. The set $\wt
\calA_{\pi,X}$ is a single $G_\pi(\A_f)$-orbit and it is isomorphic to
$G_\pi(\A_f)/U_X$ after choosing a base point. The projection map
$\pr_\calA$ is simply modulo the left action of
$G_\pi(\Q)$. Therefore, we obtain a bijection
\[ \calA_{\pi,X} \isoto G_\pi (\Q)\backslash G_\pi(\A_f)/U_X. \]
Running over all elements $X$ in $\grX_{\pi}$, we obtain 
the following result.


\begin{thm}\label{1}
There is a bijection 
  \begin{equation}
    \label{eq:thm1}
   \calA_\pi \isoto \coprod_{X\in \grX_{\pi}} 
   G_\pi (\Q)\backslash G_\pi(\A_f)/U_X. 
  \end{equation}
\end{thm}

For each $X=([M], [T_\ell]_{\ell\neq p})\in \grX_\pi$, there is an
$\Z$-order $O_X$ in $E_\pi$ such that $O_X\otimes \Zp\simeq
\End_{\rm 
  DM}(M)$ and $O_X\otimes \Z_\ell\simeq \End_{R_\ell}(T_\ell)$ for
all $\ell \neq p$. We may choose the orders 
$O_X$ so that they are contained
in a common maximal order $O_{\max}$, because any two maximal orders
are locally conjugate. Let $\Cl(O_X)$ denote the set of
isomorphism classes of locally free right $O_X$-ideals in $E_\pi$,
 and $h(O_X):=|\Cl(O_X)|$ the class number of $O_X$. 
We have $\Cl(O_X)\simeq G_\pi (\Q)\backslash
G_\pi(\A_f)/U_X$. Theorem~\ref{1} then gives
\begin{equation}
  \label{eq:12}
  |\calA_\pi|=\sum_{X\in \grX_\pi} h(O_X).
\end{equation}
Thus, to compute $|\calA_\pi|$ one needs to 
\begin{itemize}
\item[(i)] Classify members in $\grX_\pi$; 
\item[(ii)] Compute the order $O_X$ for each $X\in \grX_\pi$. This is
  again a local computation;
\item[(iii)] Compute the class number $h(O_X)$ for each $X\in
  \grX_\pi)$.      
\end{itemize}

The first step is the most complicate part. For 
$\grX_{\pi,\ell}$, this is to classify $R_\ell$-lattices in
$F_\ell^e$. The most generality of this problem is very difficult. 
We refer to the fundamental paper of Dade, Taussky and
Zassenhaus~\cite{dade-taussky-Zassenhaus} 
and subsequent work for detailed studies. 
For $\grX_{\pi,p}$, one can identify $\grX_{\pi,p}$ with the set of 
$\Fq$-points in a Rapoport-Zink space, which could be used to estimate
the size of $\grX_{\pi,p}$. 
Once the first step is done, the second step is more or less
straightforward, because one can compute $O_X$ by 
Tate's theorem for homomorphisms of abelian varieties over finite
fields. The calculation of $h(O_X)$ differs 
whether or not the endomorphism 
algebra $E_\pi$ satisfies the Eichler condition. For our case, 
this is the case if and only if $F$ is a CM field. If the Eichler
condition holds for $E_\pi$, 
then the computation of $h(O_X)$ is much simpler by
a result of Jacobinski \cite[Theorem 2.2]{jacobinski:acta1968}; 
see Section~\ref{sec:2.2} for a simplified argument using Galois 
cohomology.

Recall that a central simple algebra $\grA$ over a number field $K$
satisfies the Eichler condition (respect to the Archimedean places) if
the group of reduced norm one $(\grA\otimes_{\Q}
\R)^\times_1$ is not compact. If $\grA$ does not satisfy the
Eichler condition, then $K$ must be totally real and $\grA$ is a
totally definite quaternion $K$-algebra. In this case, the standard
tool for computing $h(O_X)$ is Eichler's trace formula. Note that the
order $O_X$ appearing in the content of abelian varieties may
not be an $O_F$-order. Thus, the Eichler trace formula developed in 
\cite{korner:1987} for an arbitrary $O_F$-order is not 
sufficiently general for computing $h(O_X)$.  
In~\cite{xue-yang-yu:ECNF} we generalize the Eichler trace formula for an
{\it arbitrary $\Z$-order} in any totally definite quaternion over a
totally real field. For the reader's convenience, we describe this
extended formula in Section~\ref{sec:gener-eichl-trace}.


  
\subsection{}\label{sec:2.2}
Assume that the center $F$ of $E_\pi$ has no real place, i.e. $\pi$
 is not $\pm \sqrt{q}$.
Let $\Nr: G_\pi \to \Res_{F/\Q} \Gm$ be the reduced norm map with kernel
$G_{\pi,1}$, and $\wh O_F$ the profinite completion of $O_F$. Define
\[ N(\pi):=\sum_{X\in \grX_{\pi}} n_X, \quad \text{where}\quad  n_X:=[\wh O_F^{\times} : O_F^\times  \Nr(U_X)]. \]


\begin{prop}\label{2}
  Assume that the field $F=\Q(\pi)$ has no real place. Then 
  $ |\calA_{\pi}|=N(\pi) h(F)$, 
  where $h(F)$ denotes the class number of the number field $F$. 
\end{prop}
\begin{proof}  
  Since the group 
  $G_{\pi_,1}$ is semi-simple and simply-connected, 
  we 
  have $H^1(\Qp,G_{\pi,1})=1$ for all primes $p$ by Kneser's Theorem 
  \cite[Theorem 6.4, p.~284]{platonov-rapinchuk}.  
  It then follows that $\Nr(G_\pi(\A_f))=\A_{F,f}^\times$.
  On the other hand, since $F$ has no real place, the Lie group
  $G_{\pi,1}(\R)$ is not compact and strong approximation holds for
  $G_{\pi,1}$. It then follows (see \cite[Lemma 14]{yu:swan} for the
  argument) that 
the reduced norm map 
\[ \Nr:
 G_\pi (\Q)\backslash G_\pi(\A_f)/U_X \isoto  \Nr(G(\Q)) \backslash
 \A_{F,f}^{\times} /\Nr(U_X) \]
is bijective. By the Hasse-Schilling-Maass norm theorem
\cite[Theorem~33.15]{reiner:mo} that every element in $F^\times$ is a
 norm if and only if it is a local norm everywhere, 
one gets $\Nr(G(\Q))=F^\times$ by Kneser's Theorem again. 
This proves 
\[ |G_\pi (\Q)\backslash
 G_\pi(\A_f)/U_X|=|\A_{F,f}^{\times}/ F^\times\cdot \Nr(U_X)|. \]  
On the other hand, consider the short exact sequence
\[ 1\longrightarrow \frac{\wh O_F^\times}{(\wh O_F^\times\cap F^\times
 \Nr(U_X))}
 \longrightarrow
 \frac{\A_{F,f}^{\times}}{F^\times \Nr(U_X)} 
\longrightarrow \Pic(O_F) \longrightarrow 1. \]   
It is easy to check 
$\wh O_F^\times\cap F^\times \Nr(U_X)=O_F^\times \Nr(U_X)$. 
Thus, we have $|G_\pi (\Q)\backslash G_\pi(\A_f)/U_X|=n_X h(F)$. 
The proposition then follows from Theorem~\ref{1}. 
\end{proof}

\subsection{An example} Take $\pi=\sqrt{-p}$. We have
$E_\pi=F=\Q(\sqrt{-p})$ and $R=\Z[\sqrt{-p}]$. The corresponding
abelian variety $A_\pi$ is a supersingular elliptic curve over
$\Fp$. The $F$-isocrystal $N(A_\pi)$ is a free $F_p$-module of rank
one, and \dieu modules in $N(A_\pi)$ are simply $R_p$-submodules. Since
$R_p$ is a maximal order, $\grX_{\pi,p}$ has one element.

When $\ell\neq 2$ or $p\equiv 1
\pmod 4$, the order $R_\ell$ is maximal and hence 
the set $\grX_{\pi,\ell}$ has one element. Therefore, when $p\equiv 1
\pmod 4$ the set $\grX_\pi$ consists of one element $X$ with $n_X=1$,
and $N(\pi)=1$. 

Suppose that  $p\equiv 3 \pmod 4$. The set $\grX_{\pi,2}$
consists of two elements:
\[ T_2\simeq R_2,\quad \text{or} \quad  T_2\simeq O_{F,2}. \]
Therefore, the $\grX_{\pi}$ has two elements $X_1$ and $X_2$,
corresponding to the two elements in $\grX_{\pi,2}$ above.   
We have $n_{X_2}=1$ and 
$n_{X_1}=[\wh O_F^{\times}:O_F^\times
\wh R^\times]$. One computes that \cite{yu:sp-prime}
\[ [\wh O_E^\times:O_E^\times \wh R^\times]=
\begin{cases}
  1, & \text{if $p\equiv 7 \pmod 8$ or $p=3$},\\
  3, & \text{if $p\equiv 3 \pmod 8$ and $p\neq 3$}. 
\end{cases} \]
Therefore, $|\calA_\pi|=N(\pi) h(\Q(\sqrt{-p}))$, where
\[ N(\pi)=
\begin{cases}
  1, & \text{if $p\equiv 1 \pmod 4$ or $p=2$},\\
  2, & \text{if $p\equiv 7 \pmod 8$ or $p=3$},\\
  4, & \text{if $p\equiv 3 \pmod 8$ and $p\neq 3$}. 
\end{cases} \]
In \cite{schoof:1987} one finds more formulas for numbers of elliptic
curves in an isogeny class over $\Fq$. 

\subsection{}
Proposition~\ref{2} basically says that when $F$ has no real place, the
computation of $|\calA_\pi|$ is reduced to the problem of
computing the class numbers of certain orders $R_X$ in $F$ 
for each $X\in \grX_\pi$, and we have shown the relation 
$h(R_X)=n_X h(F)$. Let $\wt n_X:=[\hat O_F^\times, \Nr(U_X)]$, then we
have $\wt n_{X}=\prod_{\ell} \wt n_{X_\ell}$, where
$\wt n_{X_\ell}:=[\hat O_{F_\ell}^\times: \Nr(U_{X,\ell})]$. 
Put $\wt N_\ell(\pi):=\sum_{X_\ell\in \grX_{\pi,\ell}} \wt n_{X_\ell}$
and $\wt N(\pi):=\sum_{X\in \grX_{\pi}} \wt n_{X}$. Then we have an
upper bound which can be computed locally
\begin{equation}
  \label{eq:2.1}
  N(\pi)\le \wt N(\pi) \quad\ \text{and}\ \wt N(\pi)=\prod_{\ell} 
\wt N_{\ell}(\pi). 
\end{equation}


We make a few remarks on the classification of $\grX_{\pi,\ell}$,
i.e.~classifying $R_\ell$-modules in a free $F_\ell$-module of rank
$n$. 

First observe that $R=\Z[\pi]$ is a complete intersection and
Gorenstein. It might be possible to study $R_\ell$-modules through
(co)-homological algebra. 

In some special cases, it is possible to describe the finite set
$\grX_{\pi,\ell}$ more explicitly. For example if $R[1/p]$ is a Bass
($\Z[1/p]$-)order, then one has a more explicit description of 
$R_\ell$-modules (see \cite[Section 37]{curtis-reiner:1} for the
definition of Bass orders).
In this case any $R_\ell$-module $T_\ell$ in
$F_\ell^n$ is isomorphic to $R_1\oplus R_2\oplus \dots \oplus R_n$ 
for orders $R_1\subset \dots \subset R_n$ containing $R_\ell$ in
$F_\ell$, and   
 the set $\{R_1,\dots, R_n\}$ of orders with multiplicities 
is completely determined by $T_\ell$. 

Examples of Bass orders include Dedekind domains and quadratic orders
over a Dedekind domain. Bass orders share the local property in the
sense that an order is Bass if and only if so are all its
completions. 

We illustrate the idea by an example. Assume that $F=\Q(\pi)$ is an
imaginary quadratic field and let $n=\dim A_\pi$. For any integer
$d\ge 1$, denote by $R_d$  the order in $F$ of conductor $d$. 
Suppose the conductor of $R=\Z[\pi]$ is $q_1 D$,
where $q_1$ is a $p$-power and $(p,D)=1$. For any $\Z$-lattice $L$
write $\wh L^{(p)}:=L\otimes \wh \Z^{(p)}$ with
$\Z^{(p)}=\prod_{\ell\neq  p}\Z_\ell$. 
Then any 
$\wh R^{(p)}$-module  $T^{(p)}$ of rank $n$ is isomorphic
to $\wh R_{d_1}^{(p)} \oplus \wh R_{d_2}^{(p)} \oplus \dots \oplus
\wh R_{d_n}^{(p)}$ for uniquely determined divisors $d_1,\dots, d_n$ of 
$D$ with $d_n|d_{n-1}|\dots |d_1$.  One computes directly that $\Nr(\End(T^{(p)}))=\wh R_{d_n}^{(p)}$. 
Suppose that we
have found the representatives $M_1, \dots, M_r$ for
$\grX_{\pi,p}$. Then $\Nr(\End_{\rm DM}(M_i))=(R_{p^{a_i}})_p$ for
some non-negative integers $a_i$. The cardinality of $\calA_\pi$ is
given by
\begin{equation}
  \label{eq:14}
  |\calA_{\pi}|=\sum_{i=1}^r \sum_{d_1,\dots, d_n} h(R_{p^{a_i}
   d_n}), 
\end{equation}
where $d_i$'s run over the positive divisors of $D$ with the condition
$d_n|d_{n-1}|\dots |d_1$.

\subsection{}\label{subsec:sqrt-q} Assume that $F$ has a real place. Then $\pi=\pm
  \sqrt{p^a}$, $F=\Q(\sqrt{p^a})$ and $R=\Z[\sqrt{p^a}]$. We separate
  the discussion into two cases depending on the parity of
  $a$. 
  
When $a$ is even, $F=\Q$ and $E_\pi$ is isomorphic to the quaternion
  $\Q$-algebra 
  $D_{p, \infty}$ ramified exactly at $\{p, \infty\}$. The set $\grX_\pi$
  consists of one element $X$ and the corresponding group $U_X$ is
  maximal in $G_{\pi}(\A_f)$. Thus, $|\calA_\pi|$ is the class
  number of a maximal order of $D_{p,\infty}$. 
  The class number formula for $D_{p,\infty}$ 
  due to Deuring, Eichler and Igusa gives
\[ 
|\calA_{\pi}|=\frac{p-1}{12}
+\frac{1}{3}\left (1-\left (\frac{-3}{p}\right)\right)
+\frac{1}{4}\left(1-\left(\frac{-4}{p}\right)\right).
\] \

Now consider the case where $a$ is odd. 
The endomorphism algebra $E_\pi$ is isomorphic to the quaternion algebra 
$D_{\infty_1,\infty_2}$ over $F=\Q(\sqrt{p})$ ramified exactly at the two
Archimedean places $\{\infty_1,\infty_2\}$. Moreover, the corresponding
abelian variety $A_\pi$ is a supersingular abelian surface over
$\Fq$. 



Note that $[O_F:R]$
is a divisor of $2p^{(a-1)/2}$.
So for $\ell\nmid
2p$, the ring $R_\ell$ is the maximal order and
$\grX_{\pi,\ell}$
has one element $T_\ell$
whose endomorphism ring is isomorphic to $\Mat_2(O_{F,\ell})$, where
$O_{F, \ell}:=O_F\otimes \Z_\ell$. 
For an odd prime $p$, the projection gives natural identification
$\grX_\pi=\grX_{\pi,p}\times
\grX_{\pi,2}$. Similarly, if $p=2$, then $\grX_\pi=\grX_{\pi,p}$.

  When $p\equiv 3 \pmod 4$, the order $R_2$ is maximal and
  hence $\grX_{\pi,2}$ has one element whose endomorphism ring is
  $\Mat_2(O_{F,2})$.  When
  $p\equiv 1 \pmod 4$, the set $\grX_{\pi, 2}$ has three elements
  $[L_1]$, $[L_2]$, $[L_4]$ where
\[ L_1=O_{F,2}^{\oplus 2}, \quad L_2=R_2\oplus O_{F,2}, \quad
\text{and} \quad L_4=R_2^{\oplus 2}. \]
The corresponding endomorphism rings are $\End_{R_2}(L_1)= 
\Mat_2(O_{F,2})$,
\[  \End_{R_2}(L_2)=
\begin{pmatrix}
  R_2 & 2O_{F,2} \\
  O_{F,2} & O_{F,2} 
\end{pmatrix}, \quad \text{and}\quad \End_{R_2}(L_4)=\Mat(R_2). \]
The latter two rings have index 8 and 16 respectively in
$\Mat_2(O_{F,2})$.
Let us fix a maximal ring $\bbO_1\subseteq D_{\infty_1, \infty_2}$ and
identify $\bbO_1\otimes \Z_2$ with $\End_{R_2}(L_1)$. If
$p\equiv 1\pmod{4}$, then we define $\bbO_8\subset \bbO_1$ to be the
unique suborder of index $8$ such that
$\bbO_8\otimes \Z_2=\End_{R_2}(L_2)\subset\End_{R_2}(L_1)$, and
$\bbO_{16}\subseteq \bbO_1$ to be the unique suborder of index $16$
such that $\bbO_8\otimes \Z_2=\End_{R_2}(L_4)$.

We partition set $\grX_{\pi,p}$ into disjoint subsets
$\grX_{\pi,p}^1\sqcup \grX_{\pi,p}^2$ according to the
$\mathbf{a}$-number of its members (see
\cite[Section 1.5]{li-oort}), where
$\grX_{\pi,p}^i:=\{[M]\in \grX_{\pi,p}\mid \mathbf{a}(M)=i\}$ for
$i=1, 2$. Accordingly, $\grX_\pi=\grX_\pi^1\sqcup \grX_\pi^2$ and
$\calA_\pi=\calA_\pi^1\sqcup \calA_\pi^2$.  The subset $\calA_\pi^2$
is none other than the set of \textit{superspecial} abelian surfaces
up to isomorphism in $\calA_\pi$ (Section~1.7, ibid.), so we also
denoted it by $\Sp(\pi)$.  Since $a$ is odd, $\grX_{\pi,p}^2$ is a
singleton $\{[M_0]\}$, with endomorphism ring $\End_{\rm DM}(M_0)$
isomorphic to the maximal order $\Mat_2(O_{F,p})$. More explicitly,
$M_0=R_p^{\oplus 2}\otimes_{\Z_p} W(\F_q)$, and the Frobenius map acts
as $\sqrt{p}\sigma$, where $\sigma$ is the Frobenius morphism of
the ring of Witt vectors $W(\F_q)$. 

Let $[A]$ be a member of $\Sp(\pi)$.  If $p=2$ or $p\equiv 3\pmod{4}$,
then $\End_{\F_q}(A)$ is  maximal in $D_{\infty_1, \infty_2}$. If
$p\equiv 1\pmod{4}$, then $\End(A)$ is locally isomorphic to $\bbO_1$,
$\bbO_8$ or $\bbO_{16}$, depending on the component of $\Phi([A])$
in $\grX_{\pi,2}$. We conclude that 
\begin{equation}\label{eq:10}
\absv{\Sp(\pi)}=
\begin{cases}
  h(\bbO_1) &\quad \text{if } p=2 \text{ or } p\equiv 3\pmod{4},\\
  h(\bbO_1)+h(\bbO_8)+h(\bbO_{16}) &\quad \text{if } p\equiv 1\pmod{4}.\\
\end{cases}
\end{equation}
The cardinality of $\Sp(\pi)$ is explicitly calculated in
\cite{xue-yang-yu:sp_as}, and that of $\calA_\pi^1$ will be calculated
in Section~\ref{sec:ssing-abel-vari}.

\section{Eichler's trace formula for Brandt Matrices}
\label{sec:gener-eichl-trace}

In this section we extend the classical notion of Brandt matrices
\cite[Exercise~III.5.8]{vigneras} to arbitrary $\bbZ$-orders in
totally definite quaternion algebras and provide a trace formula for
them. As a result, we obtain a class number formula for all such
$\bbZ$-orders. Details of the proofs may be found in
\cite{xue-yang-yu:ECNF}. 


Throughout this Section, $F$ denotes a totally real number field,
$A\subseteq O_F$ a $\bbZ$-order in $F$, and $D$ a totally definite
quaternion $F$-algebra. A $\bbZ$-order $\calO\subset D$ is said to be
a \textit{proper} $A$-order if $\calO\cap F=A$.  Similarly, we define
the notion of proper $A$-orders in finite field extensions
$K/F$.  For any $A$-lattice $I\subset D$, the norm $\Nr_A(I)$ of $I$
over $A$ is definite to be the $A$-submodule of $F$ generated by the
reduced norms of elements of $I$. 
If the multiplication $IJ$ of two $A$-lattices $I$ and $J$ is
\textit{coherent} \cite[Section~I.4]{vigneras} and one of $I$ and $J$
is locally principal with respect to its associated left (or right) order, then
$\Nr_A(IJ)=\Nr_A(I)\Nr_A(J)$.

When $I=\calO$, the norm $\Nr_A(\calO)$ is an $A$-order in $F$,
denoted by $\wt A$. It is known that $\wt{A}=A$ if and only if $\calO$
is closed under the canonical involution $x\mapsto \Tr(x)-x$
(\cite[Lemma~3.1.1]{xue-yang-yu:ECNF}).  If $I$ is a locally principal
right $\calO$-ideal, then $\Nr_A(I)$ is a locally principal
$\wt A$-ideal.  Let $\Cl(\calO)$ be the set of isomorphism classes of
locally principal right $\calO$-ideals in $D$, and
$h=h(\calO)=\lvert\Cl(\calO)\rvert$ the class number of $\calO$. We
fix a complete set of representatives $I_1,\dots, I_h$ for $\Cl(\calO)$, and set
\begin{equation}
  \label{eq:34}
 \calO_i:=\calO_l(I_i),\qquad  w_i:=[\calO_i^\times: A^\times]. 
\end{equation}
Each $\calO_i$ is a proper $A$-order uniquely determined up to
$D^\times$-conjugation, and $w_i$ depends only on the ideal class of
$I_i$. Note that $\Nr_A(\calO_i)=\wt A$  for all $1\leq i\leq h$. 
  The \textit{mass} of $\calO$ is defined as the weighted sum
\[ \Mass(\calO):=\sum_{i=1}^h\frac{1}{[\calO_i^\times
  :A^\times]}=\sum_{i=1}^h \frac{1}{w_i}. \]

\begin{defn}\label{defn:brandt-matrix}
  Let $\grn$ be a locally principal integral $\wt A$-ideal.  The
  \textit{Brandt matrix associated to $\grn$} is defined to be the
  matrix
$\grB(\grn):=(\grB_{ij}(\grn))\in \Mat_h(\bbZ),  $
where $\grB_{ij}(\grn)$ is the cardinality of the set of right $\calO_j^\times$-orbits of elements
  $b\in I_i I_j^{-1}$ such that $\Nr_A(b\calO_j)=\grn \Nr_A(I_i)\Nr_A(I_j)^{-1}$.
\end{defn}


It is clear from the Definition~\ref{defn:brandt-matrix} that
$\grB_{ii}(\grn)\neq 0$ only if $\grn$ is principal and generated by a
totally positive element. So for the rest of this section we assume
that $\grn=\wt{A}\beta$ is generated by a totally positive element
$\beta\in \wt{A}$. Moreover, we define the symbol
\begin{equation}
\label{eq:2}
 \delta_{\grn}=
\begin{cases}
  1 & \text{if $\grn=\wt{A}a^2$ for some $a\in A$}; \\
  0 & \text{otherwise.}
\end{cases}
\end{equation}

A finite $A$-order (i.e. a finite $A$-algebra with no $\Z$-torsion)
$B$ is said to be a \textit{CM proper $A$-order} if
$K:=B\otimes_\bbZ\bbQ$ is a CM-extension of $F$
(\cite[Section~13]{Conner-Hurrelbrink}), and $B$ is a proper $A$-order
in $K$. We set
$\delta(B)=1$ if $B$ is closed under the complex conjugation of $K/F$,
and $\delta(B)=0$ otherwise.  Denote by $\Emb(B,\calO)$ the set of
optimal $A$-embeddings from $B$ into $\calO$, i.e.,
\[\Emb(B,\calO):=\{\varphi\in \Hom_F(K, D)\mid \varphi(K)\cap \calO=\varphi(B)\}.\]
This is a finite set equipped with a right action of
$\calO^\times$ sending $\varphi\mapsto g^{-1}\varphi g$ for all
$\varphi\in \Emb(B, \calO)$ and $g\in \calO^\times$.
We denote
\begin{equation}
  \label{eq:3}
m(B,\calO,\calO^\times):= \absv{
 \Emb(B,\calO)/\calO^\times}, \quad \text{and}\quad
w(B):=[B^\times: A^\times].
\end{equation}
Similarly for each prime $p$, we put
\begin{equation}
  \label{eq:4}
m_p(B):=m(B_p,\calO_p,\calO_p^\times)= \absv{
  \Emb(B_p,\calO_p)/\calO_p^\times},  
\end{equation}
where $B_p$ and $\calO_p$ denote the $p$-adic completions
$B\otimes \bbZ_p$ and $\calO\otimes \bbZ_p$, respectively. Note that
$m_p(B)=1$ for all but finitely many $p$. 
Choose a complete set
$S=\{\epsilon_1,\dots, \epsilon_s\}$ of representatives for the finite
group $\wt A^\times_+/(A^\times)^2$, where $\wt A^\times_+$ denotes
the subgroup of totally positive elements in $\wt A^\times$. Let $T_{B,\grn}\subset B$ be the finite set
  \begin{equation}\label{eq:5}
    T_{B, \grn}:=\{x\in B\setminus A\, |\,
  N_{K/F}(x)=\varepsilon\beta \ \ \text{for some $\varepsilon\in S$}\,
  \}. 
  \end{equation}
  For each fixed $\grn$, there are only finitely many CM proper
  $A$-orders $B$ with $T_{B, \grn}\neq
    \emptyset$. 

  \begin{theorem}[Eichler's Trace Formula]\label{thm:trac-brandt-matr}
    Suppose that $\grn=\wt{A}\beta$ is generated by a totally positive
    element $\beta\in \wt{A}$.  Then the trace of the Brandt matrix
    $\grB(\grn)$ is given by
    \[\Tr\grB(\grn)=\delta_\grn \cdot \Mass(\calO)+\frac{1}{4}
    \sum_{B}\frac{(2-\delta(B))h(B)\absv{T_{B,\grn}}}{w(B)} \prod_{p}
    m_p(B), \]
    where $B$ runs over all mutually non-isomorphic CM proper $A$-orders with $T_{B, \grn}\neq
    \emptyset$. 
  \end{theorem}
  The proof of this theorem follows closely Eichler's original proof
  \cite{eichler:crelle55}; see also Vign\'eras's book \cite{vigneras}.
 When $\grn=(1)=\wt{A}$, the Brandt matrix $\grB(\wt{A})\in
  \Mat_h(\bbZ)$ is the 
 identity and $\Tr \grB(\wt{A})=h(\calO)$. 

 \begin{cor}[Eichler's Class number
   formula] \label{class_number_formula} Let $\calO$ be a proper
   $A$-order in a totally definite quaternion algebra $D/F$. Then
\begin{equation}\label{eq:6}
  h(\calO)=\Mass(\calO)+\frac{1}{2} \sum_{ w(B)>1}(2-\delta(B))
  h(B)(1-w(B)^{-1})\prod_{p} m_p(B), 
\end{equation}
where $B$ runs over all mutually non-isomorphic CM proper
$A$-orders with $w(B)=[B^\times:A^\times]>1$. 
\end{cor}

Let $h(\sqrt{d})$ denotes the class number of $\Q(\sqrt{d})$ for any
square free $d\in \bbZ$.  Applying formula (\ref{eq:6}) to the orders
$\bbO_8$ and $\bbO_{16}$ defined in Subsection~\ref{subsec:sqrt-q}
when $p\equiv 1\pmod{4}$, we obtain 
\begin{align*}
  h(\mathbb{O}_8)&=\varpi_ph(\sqrt{p})\left[\left(4-\Lsymb{2}{p}\right)
  \frac{\zeta_F(-1)}{2}+\left(2-\Lsymb{2}{p}\right)\frac{h(\sqrt{-p})}
  {24}+\frac{\delta_{1,\varpi_p}h(\sqrt{-3p})}{3} 
  \right],  \\ 
  h(\mathbb{O}_{16})&=\varpi_ph(\sqrt{p})\left[ \left(3-2\Lsymb{2}{p} 
  \right)\zeta_F(-1)+\left(2-\Lsymb{2}{p}\right)\frac{h(\sqrt{-p})}{12} 
  +\frac{1}{6}h(\sqrt{-3p})\right]. 
\end{align*}
Here
$\varpi_p:=3[O_{\Q(\sqrt{-p})}^\times:\bbZ[\sqrt{p}]^\times]^{-1}\in
\{1,3\}$,
and $\delta_{1, \varpi_p}=1$ or $0$ depending on whether $\varpi_p=1$
or $3$. When $p\equiv 1\pmod{8}$, we always have $\varpi_p=3$ by
\cite[Lemma~4.1]{xue-yang-yu:num_inv}, and hence
$\delta_{1,\varpi_p}=0$ in this case.  The special value $\zeta_F(-1)$
of the
Dedekind zeta-function can be calculated by Siegel's formula
\cite[Table~2, p.~70]{Zagier-1976-zeta}:
\begin{equation}
  \label{eq:26}
  \zeta_F(-1)=\frac{1}{60}\sum_{\substack{b^2+4ac=\grd_F\\ a,c>0}} a,
\end{equation}
where  $b\in \Z$ and $a,c\in \bbN_{>0}$.
We refer to \cite{xue-yang-yu:ECNF} for the details of the calculation
of the above 
formulas. 

For the sake of completeness, we also list the class number of 
the maximal
order $\bbO_1$.
If $p\equiv 1\pmod{4}$ and $p>5$, then
\begin{align*}
  h(\bbO_1)&=h(\sqrt{p})\left[\frac{\zeta_F(-1)}{2} +
     \frac{h(\sqrt{-p})}{8}+\frac{h(\sqrt{-3p})}{6}\right].      
\end{align*}
If $p\equiv 3\pmod{4}$ and $p\geq 7$, then
\begin{equation*}
  \begin{split}
  h(\bbO_1)&=h(\sqrt{p})\left[  \frac{\zeta_F(-1)}{2} +
     \left(\frac{13}{8}-\frac{5}{8}\left(\frac{2}{p}
    \right)\right)h(\sqrt{-p}) 
     +\frac{h(\sqrt{-2p})}{4}+\frac{h(\sqrt{-3p})}{6}  \right].    
  \end{split}
\end{equation*}
Lastly, $h(\bbO_1)=1, 2, 1$ for
the primes $p=2, 3, 5$, respectively.  These formulas can be
calculated directly using the \textit{classical} Eichler  class
number formula \cite[Corollaire~2.5]{vigneras}. 

\section{Supersingular abelian surfaces in an isogeny class}
\label{sec:ssing-abel-vari}
In this section we compute $|\calA_\pi|$ for $\pi=\sqrt{q}$, where
$q=p^a$ is an odd power of $p$. 
Recall that two orders $\calO_1$ and $\calO_2$ in a quaternion algebra are said to
be in the same \textit{genus} if their $p$-adic completions are
isomorphic for all primes $p$.  Let $D=D_{\infty_1, \infty_2}$ be the totally definite
quaternion algebra over $F=\Q(\sqrt{p})$ that splits at all finite
places. Recall that the orders $\bbO_8$ and $\bbO_{16}$ in $D$ are
defined only when $p\equiv 1\pmod{4}$.  Let $\calO$ be an order in the
genus of $\bbO_r$ with $r=1, 8$ or $16$ and $\grR:=\calO\cap F$ the
center of $\calO$. Then $\grR=O_F$ if $\calO$ is maximal, and
$\grR=\Z[\sqrt{p}]$ otherwise. Nevertheless, $\grR_p$ is always the maximal
order in $F_p$.  The $p$-adic completion
$\calO_p:=\calO\otimes \Z_p\simeq \Mat_2(\Z_p[\sqrt{p}])$ is a maximal
order in $D\otimes \Q_p$, and the natural projection
\[\calO\to \calO/\sqrt{p}\calO\cong \calO_p/\sqrt{p}\calO_p\simeq
\Mat_2(\F_p)\] induces a 
group homomorphism 
\begin{equation}\label{eq:7}
 \rho: \calO^\times/\grR^\times\to \PGL_2(\F_p). 
\end{equation}

Let $\tilde{u}\in \calO^\times/\grR^\times$ be a nontrivial element, and
$u\in \calO^\times$ a representative of $\tilde{u}$. Both the field
$F(u)\subset D$ and its suborders $\grR[u]\subseteq F(u)\cap \O$ depends
only on $\tilde{u}$, not on the choice of $u$. For simplicity,  we set 
\[K_{\tilde{u}}:=F(u),\quad \grR[\tilde{u}]:=\grR[u], \quad\text{and}\quad
B_{\tilde{u}}:=K_{\tilde{u}}\cap \calO. \]
Then both $B_{\tilde{u}}$ and $\grR[\tilde{u}]$ are CM proper $\grR$-orders
with $[B_{\tilde{u}}:\grR^\times]\geq [\grR[\tilde{u}]^\times:
\grR^\times]>1$. All such orders $\grR[\tilde{u}]$ have been classified in
\cite{xue-yang-yu:num_inv}. 
\begin{lemma}\label{lem:kernel-criterion}
  If the index $[B_{\tilde{u}}:\grR[\tilde{u}]]$ is coprime to $p$, then
  $\tilde{u}\not\in \ker(\rho)$.
\end{lemma}
\begin{proof}
Suppose that $p\nmid   [B_{\tilde{u}}:\grR[\tilde{u}]]$.  Then
$B_{\tilde{u}}\otimes \Z_p=\grR_p[u]$, and 
\[\calO/\sqrt{p}\calO\supset B_{\tilde{u}}/\sqrt{p}B_{\tilde{u}}=\grR_p[u]\otimes_{\grR_p}
(\grR_p/\sqrt{p}\grR_p)=\F_p[\bar{u}],\]
where $\bar{u}$ denotes the image of $u$ modulo
$\sqrt{p}B_{\tilde{u}}$. Since $\grR_p$ is maximal in $F_p$, $\grR_p[u]$ is
a free $\grR_p$-module of rank $2$, and hence $\F_p[\bar{u}]$ is an
$\F_p$-space of dimension $2$. In particular, $\bar{u}\not\in \F_p$. 
\end{proof}

Recall the $h(D)=1, 2, 1$ when $p=2, 3, 5$ respectively. In
particular, when $p=2$ or $5$, $\bbO_1$ is the unique maximal order in
$D$ up to conjugation. Using the Magma Algebra System \cite{magma},
one checks that 
\begin{equation}
  \label{eq:4.2a}
  \bbO_1^\times/O_F^\times\simeq S_4 \quad \text{if }p=2, \quad
\text{and} \quad \bbO_1^\times/O_F^\times\simeq A_5 \quad \text{if
}p=5.
\end{equation}
When $p=3$, there are two maximal orders $\bbO, \bbO'$ up to
conjugation. We have 
\begin{equation}
  \label{eq:4.3}
\text{ $\bbO^\times/O_F^\times\simeq S_4$ and
  $\bbO'^\times/O_F^\times\simeq 
D_{12}$,  the dihedral group of order $24$.} 
\end{equation}
. 

\begin{prop}\label{17} \

  {\rm (1)} The morphism $\rho$ is injective for any order
  $\calO$ in the genus of $\bbO_r$ with $r=1, 8, 16$ when $p\geq 5$. 

  {\rm (2)} When $p=2$,  $\rho$ is surjective. 

  {\rm (3)} When $p=3$, $\rho$ is an isomorphism for $\calO=\bbO$, and
  $\img(\rho)\simeq D_4$ for $\calO=\bbO'$.
\end{prop}
\begin{proof}
  Let $\tilde{u}\in \calO^\times/\grR^\times$ be a nontrivial
  element. According to the classification in
  \cite{xue-yang-yu:num_inv}, if
  $\ord(\tilde{u})$ is a power of $2$, then $\ord(\tilde{u})=2$ or
  $4$, and $[O_{K_{\tilde{u}}}: \grR[\tilde{u}]]\in \{1, 2, 4\}$. By
  Lemma~\ref{lem:kernel-criterion}, $\rho$ is injective on the Sylow
  $2$-subgroups of $\calO^\times/\grR^\times$ when $p$ is odd.  Similarly,
  if $p\neq 3$ and $\ord(\tilde{u})=3$, then
  $[O_{K_{\tilde{u}}}: \grR[\tilde{u}]]\in \{1, 2, 4\}$, so once again
  $\tilde{u}\not\in \ker(\rho)$.  This exhaust the list of all
  nontrivial elements of $\calO^\times/\grR^\times$ when one of the
  following conditions holds: (i) $p\geq 7$; (ii) $p=5$ and $\calO$ is
  not maximal.  Hence $\rho$ is injective in these two
  cases. Moreover, if $p=5$ and $\calO$ is maximal, then $\rho$ is
  injective since $\calO^\times/O_F^\times$ is the simple
  group $A_5$.   Part (1) of the proposition follows.
 
  Suppose that $p=2$. Then $\calO^\times/O_F^\times \simeq S_4$, and
  for any element $\tilde{u}\in \calO^\times/O_F^\times$ of order $4$,
  $O_F[\tilde{u}]$ coincides with the maximal order
  $O_{K_{\tilde{u}}}\simeq \bbZ[\zeta_8]$. Thus
  $\rho(\tilde{u})\in \PGL_2(\F_2)\simeq S_3$ is a nontrivial element
  of order $2$. Since $\rho$ does not vanish on any order 3 elements
  as remarked, it maps $\calO^\times/O_F^\times$ surjectively onto
  $\PGL_2(\F_2)$.

  Suppose that $p=3$ and $\calO^\times/O_F^\times\simeq S_4$. In this
  case we have $D=\qalg{-1}{-1}{\Q(\sqrt{3})}$, and $\calO$ is
  $D^\times$-conjugate to $B+B(1+i+j+k)/2$, where
  $B=\Z[\sqrt{3}, (1+\sqrt{3})(1+i)/2]\subset F(i)$. If
  $\tilde{u}\in \calO^\times/O_F^\times$ is an element of order $3$,
  then $O_F[\tilde{u}]=B_{\tilde{u}}\simeq \Z[\sqrt{3}, \zeta_6]$, and
  hence $\tilde{u}\not\in \ker(\rho)$. Since $\rho$ is also injective
  on the Sylow $2$-subgroups of $\calO^\times/O_F^\times$, it induces
  an isomorphism between $\calO^\times/O_F^\times$ and
  $\PGL_2(\F_3)\simeq S_4$. 

  Lastly, suppose that $p=3$ and
  $\calO^\times/O_F^\times\simeq D_{12}$, whose Sylow 2-subgroups are
  isomorphic to $D_4$.  Note that $D_{12}$ has a unique Sylow
  $3$-subgroup, which is cyclic of order $3$, and
  $D_{12}=C_3\rtimes D_4$. Necessarily, $\ker(\rho)$ is nontrivial
  because $D_{12}\not\simeq \PGL_2(\F_3)$. Since $\rho$ is injective on
  the Sylow $2$-subgroups, we must have $\ker(\rho)=C_3$ and
  $\img(\rho)\simeq D_4$. 
\end{proof}

Let $\pi=\sqrt{q}$ be the Weil $q=p^a$-number with $a$ odd. Recall
that $\calA_\pi$ denotes the set of isomorphism classes in the
$\Fq$-isogeny class of $A_\pi$. The goal is to evaluate $|\calA_\pi|$.
We decompose $\calA_\pi=\Sp(\pi)\coprod \calA_\pi^1$ as the union of
the set $\Sp(\pi)$ of superspecial abelian surfaces and $\calA_\pi^1$
consisting of non-superspecial abelian surfaces.  The case $q=p$ is
done in \cite{xue-yang-yu:ECNF} and we may assume that $q>p$ for which
$\calA_\pi^1$ is non-empty. Let $A\in \calA_\pi^1$ be a member with
covariant \dieu module $M$.  There is a unique isogeny
$\varphi:A\to \wt A$ of degree $p$ with $\wt A$ superspecial: $\wt A$
is constructed with the \dieu module $M+F^{-1}VM$ \cite[Lemma
4.4]{yu:ss-siegel}.  Consider the map $\pr: \calA_\pi^1\to \Sp(\pi)$
sending $A$ to $\wt A$.
 
\begin{prop}\label{18} \ 

  {\rm (1)} Let $A_0\in \Sp(\pi)$ be a superspecial member. The fiber
  $\pr^{-1}(A_0)$ in $\calA_\pi^1$ has $(q-p)/|\rho(\Aut(A_0)|$
  elements, where $\rho$ is the map defined in \eqref{eq:7} by viewing
  the endomorphism ring $\End(A_0)$ as an order in $D$.

{\rm (2)} We have $|\calA_{\pi}|=|\Sp(\pi)|+\sum_{A_0\in \Sp(\pi)}
(q-p)/|\rho(\Aut(A_0)|$.

\end{prop}
\begin{proof}
  (1) Let $M_0$ be the \dieu module of $A_0$. If $A\in \pr^{-1}(A_0)$,
  then the \dieu module $M$ of $A$ satisfies (i) 
  $VM_0\subset M \subset M_0$, (ii) $\ol M:=M/VM_0$ is $1$-dimensional
  $\Fq$-subspace of $\ol M_0:=M_0/VM_0$, 
  and (iii) $M$ is non-superspecial. 
  Conversely, such a \dieu module gives
  rise to a member $A$ in $\calA_\pi^1$ with $\wt A\simeq A_0$. If
  two members $\varphi_i:A_i\to A_0$ ($i=1,2$) of $\pr^{-1}(A_0)$ are
  isomorphic, then any isomorphism $\alpha:A_1\isoto A_2$ lifts
  uniquely to an automorphism $\wt \alpha:A_0\isoto A_0$. Thus,
  $\pr^{-1}(A_0)$ is in bijection with the set $\calX$ of \dieu
  modules $M$ 
  satisfying (i),(ii) and (iii) modulo the action of $\Aut(A_0)$. 
  The set $\calX$ consists of points $\ol M$ of $\bfP^1_{\ol M_0}(\Fq)$
  satisfying (iii). By \cite[Lemma 6.1]{yu:ss-siegel},  
  that a point $\ol M\in \bfP^1_{\ol M_0}(\Fpbar)$ is superspecial if
  and only if it is contained in $\bfP^1_{\ol M_0}(\F_{p^2})$. It 
  follows that $\calX\simeq \bfP^1_{\ol M_0}(\Fq)-\bfP^1_{\ol
  M_0}(\Fp)$. Since $A_0$ is superspecial, $VM_0=\sqrt{p}M_0$ and the
  action of $\Aut(A_0)$ 
  on $\calX$ factors through the action on $\ol
  M_0=M_0/\sqrt{p}M_0$. This shows that 
  \begin{equation}
    \label{eq:4.2}
   \pr^{-1}(A_0)\simeq
  \Aut(A_0)\backslash \calX \simeq \rho(\Aut(A_0))\backslash
  [\bfP^1_{\ol M_0}(\Fq)-\bfP^1_{\ol M_0}(\Fp)]. 
  \end{equation}
  The action of $\PGL_2(\Fp)$ on $\bfP^1(\Fq)$ is faithful and 
  it has no
  fixed points in $\bfP^1(\Fq)-\bfP^1(\Fp)$. Therefore,
  $|\pr^{-1}(A_0)|=(q-p)/|\rho(\Aut(A_0)|$.         
  (2) This follows  from (1). 
\end{proof}

Note that the formula for $|\calA_\pi|$ holds when $q=p$. 

\begin{theorem}\label{4.4} 
Let $\pi=\sqrt{q}$, where $q$ is an odd power a prime $p$. 

{\rm (1)} For $p=2$, $|\calA_\pi|=1+(q-2)/6$.

{\rm (2)} For $p=3$, $|\calA_\pi|=2+(q-3)/6$.

{\rm (3)} For $p=5$, $|\calA_\pi|=3+4(q-5)/15$.

{\rm (4)} For $p> 5$ and $p\equiv 3 \pmod 4$, one has
\begin{equation*}
     |\calA_\pi|=h(\sqrt{p})\left[(q-p+1) \frac{\zeta_F(-1)}{2} +
      \left(\frac{13}{8}-\frac{5}{8}\left(\frac{2}{p}
    \right)\right)h(\sqrt{-p})
    +\frac{h(\sqrt{-2p})}{4}+\frac{h(\sqrt{-3p})}{6}  \right].
\end{equation*}
For $p> 5$ and $p\equiv 1 \pmod 4$, one has   
\begin{equation*}
     |\calA_\pi|=h(\sqrt{p})\left[(q-p+1)(1+5\beta_p)
     \frac{\zeta_F(-1)}{2} + 
      (1+\beta_p)\frac{h(\sqrt{-p})}{8}
   +\frac{2h(\sqrt{-3p})}{3}  \right], 
\end{equation*}
where $\beta_p=\varpi_p(2-\left(\frac{2}{p}\right ))$.
\end{theorem}
\begin{proof}
  (1) This follows form Propositions~\ref{17} (2) and \ref{18} (2), 
  $|\Sp(\pi)|=1$ and $|\PGL_2(\F_{2})|=6$. 
  (2) This follows from Propositions~\ref{17} (3) and \ref{18} (2).

  (3) and (4): By Propositions~\ref{17} (1) and \ref{18} (2),
  we have $|\calA_\pi|=|\Sp(\pi)|+(q-p)\Mass(\Sp(\pi))$. 
  The desired formulas
  follow from (\ref{eq:10}) and the class number formulas in
  Section~\ref{sec:gener-eichl-trace}. 
\end{proof}

We observe that the set $\calX$ in the proof is the set of
$\Fq$-rational points of the Drinfeld period domain $\Omega^2$ over
$\Fp$.


\section{Analogies between abelian varieties and lattices}
\label{sec:smf}



\subsection{Local analogues}
\label{sec:smf.loc}

In this section we discuss analogies between abelian varieties and
lattices. There are several descriptions of abelian varieties in terms
of ideal classes or lattices. The prototype of such descriptions 
is the Deuring-Eichler
correspondence, which establishes the following bijection
\begin{equation}
  \label{eq:smf.1}
  \left\{\parbox{1.4in}{isomorphism classes of supersingular elliptic
  curves over $\Fpbar$} 
  \right\} \oneone \left\{\parbox{1.4in}{ideal classes of a maximal
  order in 
  $D_{p,\infty}$} 
  \right\},
\end{equation}
where $p$ is a prime number 
and $D_{p,\infty}$ is the quaternion $\Q$-algebra
ramified exactly at $\{p,\infty\}$. 
We refer to 
Waterhouse~\cite{waterhouse:thesis}, 
Deligne~\cite{deligne:ord}, Ekedahl~\cite{ekedahl:ss}, 
Katsura and Oort~\cite{katsura-oort:surface},
C.-F.~Yu~\cite{yu:thesis, yu:mass-hb, yu:smf, yu:sp-prime}, Centeleghe
and 
Stix~\cite{centeleghe-stix} J.~Xue, T.-C.~Yang and C.-F. Yu
\cite{xue-yang-yu:ECNF, xue-yang-yu:sp_as, xue-yang-yu:sp_as2},
Jordan-Keeton-Poonen-Shepherd-Barron-Tate \cite{poonen_et:av} 
and others for various generalizations, 
and explicit formulas for numbers of certain abelian varieties.  
 
The basic analogy between abelian varieties and lattices 
may start as follows. 
For any $\Z$-lattice $\Lambda$, the completion $\Lambda\otimes \Z_\ell$
at a prime $\ell$ is a local analogue of $\Lambda$ at $\ell$. 
If $A$ is an abelian variety over an arbitrary field $k$, then the
$\ell$-divisible group $
A(\ell):=A[\ell^\infty]=\varinjlim A[\ell^n]$ associated to $A$
can be viewed as its local analogue at $\ell$. If $\ell \neq
\fchar k$, then there is an equivalence of categories
\begin{equation}
  \label{eq:smf.3}
  \left (\parbox{0.9in}{$\ell$-divisible groups over $k$} 
  \right ) \oneone \left (\parbox{1.6in}{finite free $\Z_\ell$-modules
  with continuous $\Gamma_k$-action} 
  \right ),
\end{equation}
where $\Gamma_k:=\Gal(k_s/k)$ is the absolute Galois group of $k$ and
$k_s$ is a separable closure of $k$. 

When $k$ is a perfect field of \ch $p>0$, the covariant \dieu theory
establishes an equivalence of categories
\begin{equation}
  \label{eq:smf.4}
  \left (\parbox{0.9in}{$p$-divisible groups over $k$} 
  \right ) \oneone \left (\, \parbox{1.5in}{finite $W$-free \dieu
  modules over $k$} \, \right ).
\end{equation}


Thus, when the ground field $k$ is finite, the local analogues of $A$
can be described using its Tate modules and \dieu module. This description
is utilized in Section~\ref{sec:isom}. We shall extend some results of  
Section~\ref{sec:isom} to more general ground fields. 


\subsection{Genera of abelian varieties }
\label{sec:smf.s}

\def\Qisog{{\rm Qisog}}
\def\GQisog{{\rm GQisog}}
\def\Gisom{{\rm Gisom}}
We shall set up a long list of definitions, in order to make
everything precise. Several of them also appear in \cite{yu:smf} and
Section 2. 
A field that is finitely generated over its prime field is called a
finitely generated field. We write $\Hom(A_1,A_2)$ for the abelian
group of $k$-homomorphisms between abelian varieties $A_1$ and $A_2$
over $k$. In the following definition, the prime $\ell$ is not
necessarily distinct from $\fchar k$.
 


\begin{defn}\label{smf.1}

\npr (1) Let $A_1$ and $A_2$ be two abelian varieties over a field
    $k$. 
    \begin{itemize}
    \item[(i)] Denote by $\Qisog(A_1,A_2)$ the set of quasi-isogenies
    $\varphi:A_1\to A_2$ over $k$, i.e. $N \varphi$ is an isogeny for
    some integer $N$.
    \item[(ii)] For any prime $\ell$, denote 
    by $\Qisog(A_1(\ell),A_2(\ell))$ the set of
    quasi-isogenies $\varphi: A_1(\ell)\to A_2(\ell)$ over $k$. Denote
    by $\Isom(A_1(\ell),A_2(\ell))$ the set of isomorphisms
    $\varphi:A_1(\ell) \isoto A_2(\ell)$ of $\ell$-divisible groups
    over $k$. 
    \end{itemize}

\npr (2) Let $x=A_0$ be an abelian variety over $k$. 
\begin{itemize}
\item[(i)] Denote by $G_x$ the automorphism group
    scheme over $\Z$ associated to $A_0$. It is the group scheme over
    $\Z$ which 
    represents the functor
    \begin{equation}
      \label{eq:smf.5}
      R\mapsto G_x(R)=(\End(A_0)\otimes_\Z R)^\times, 
    \end{equation}
    for any commutative ring $R$.
  \item[(ii)] Write $x_\ell:=A_0(\ell)$ for the
    associated $\ell$-divisible group. We definite the automorphism
    groups scheme $G_{x_\ell}$ over $\Z_\ell$ associated to 
    $A_0(\ell)$ in the same way. 
\end{itemize}
      
\npr (3) Any abelian variety $A$ in the isogeny class $[A_0]_{\rm isog}$ of
    $A_0$ can be represented by a pair $(A,\varphi)$, where $\varphi\in 
    \Qisog(A,A_0)$.
    \begin{itemize}
    \item[(i)] Denote by $\wt \calA_x$ the set of all quasi-isogenies
      $(A,\varphi)$ to $A_0$, where we regard two quasi-isogenies
      $(A_i,\varphi_i)$ to $A_0$, for
      $i=1,2$, 
      as the same member if there is an element
      $\alpha\in \Isom(A_1,A_2)$ such that
      $\varphi_1\cdot \alpha=\varphi_2$.
    \item[(ii)] Let
    $\calA_x$ be the set of isomorphism
    classes $[A]$ of all members $(A,\varphi)\in \wt
    \calA_x$. For any prime $\ell$, we define the sets $\wt
    \calA_{x_\ell}$ and $\calA_{x_\ell}$ in the same way.  
    \end{itemize}

\npr (4) Let $(A_1,\varphi_1)$ and $(A_2,\varphi_2)$ be two members in $\wt
    \calA_x$.
    \begin{itemize}
    \item[(i)] $(A_i,\varphi_i)$ ($i=1,2$) are said to be in the same {\it
        genus} if 
    $\Isom(A_1(\ell),A_2(\ell))\neq \emptyset$ for all primes
    $\ell$.
    \item[(ii)] $(A_i,\varphi_i)$ ($i=1,2$) are said to be  
   {\it (globally) equivalent} if there
    is an element $\alpha\in G_x(\Q)$ such that
    $(A_2,\varphi_2)=(A_1,\alpha \cdot \varphi_1)$. It is easy to see
    that they are equivalent if and only if their images in $\calA_x$
    are the same.

  \item[(iii)] For any member $(A,\varphi)\in \wt \calA_x$, let $\wt
    \Lambda(A,\varphi)$ denote the genus in $\wt \calA_x$ 
    that contains $(A,\varphi)$, and 
    \begin{equation}
      \label{eq:smf.6}
      \Lambda(A,\varphi):=\{[A'] \mid \exists\, \varphi' \text{ such that
         }(A',\varphi')\in \wt
      \Lambda(A,\varphi) \}. 
    \end{equation}
    
  \item[(iv)] Let
    \begin{equation}
      \label{eq:smf.7}
      G_{x,\A_f}:=\prod_{\ell}{}  ' \, G_{x_\ell}(\Q_\ell)
    \end{equation}
       be the restricted product of $G_{x_\ell}(\Q_\ell)$ with respect
        to the open compact subgroups $G_{x_\ell}(\Z_\ell)$. For any
        member $(A,\varphi)\in \wt \calA_{x}$ and any element
        $\alpha=(\alpha_\ell)\in G_{x,\A_f}$, there is a unique member
        $(A_1,\varphi_1)\in \wt \calA_x$ such that for all primes
        $\ell$, one has
        $(A_1(\ell),\varphi_{1,\ell})=(A(\ell),\alpha_\ell
        \varphi_\ell)$ 
%
        (see Lemma~\ref{smf.2}). We shall write $\alpha \cdot
        (A,\varphi)=(A_1,\varphi_1)$.    
    
      \item[(v)] Let $H\subset G_{x,\A_f}$ be a subgroup. Say
        $(A_i,\varphi_i)$ ($i=1,2$) are {\it
          $H$-equivalent} if there is an element $\alpha\in H$ such
        that $\alpha(A_1,\varphi_1)=(A_2,\varphi_2)$. It follows from
        the definition that $(A_i,\varphi_i)$ ($i=1,2$) are in the
        same genus if and only if they are $G_{x,\A_f}$-equivalent.

      \item[(vi)] Let $H$ be an algebraic group over $\Q$ together with an
        inclusion $H(\A_f)\subset G_{x,\A_f}$. Two members $(A_i,
        \varphi_i)$ ($i=1,2$) are in the same {\it $H$-genus} if they
        are $H(\A_f)$-equivalent. Thus, we have the notion of
        $G_x$-genus since $G_x(\A_f)\subset G_{x,\A_f}$. 
    \end{itemize}



\end{defn}

\begin{lemma}\label{smf.2}
  For any
        member $(A,\varphi)\in \wt \calA_{x}$ and any
        $\alpha=(\alpha_\ell)\in G_{x,\A_f}$, there is a unique member
        $(A_1,\varphi_1)\in \wt \calA_x$ such that for all primes
        $\ell$, one has
        \begin{equation}
          \label{eq:smf.8}
          (A_1(\ell),\varphi_{1,\ell})=(A(\ell),
           \alpha_\ell\cdot \varphi_\ell).
        \end{equation}
\end{lemma}
\begin{proof}
  We first prove the uniqueness of $(A_1,\varphi_1)$. 
  Note that two members $(A_1,\varphi_1)$ and
  $(A_2,\varphi_2)$ of $\wt \calA_x$  
  are the same if and only if 
  $\ker (N\varphi_1)^t= \ker (N\varphi_2)^t$ as subgroup scheme of
  $A_0^t$ for an integer $N$ such
  that $N\varphi_i$ ($i=1,2$) are isogeny. The uniqueness then is
  characterized by \eqref{eq:smf.8}. 

  Now we construct $(A_1,\varphi_1)$. Assume first that $\alpha_\ell
  \varphi_\ell$ are isogenies for all $\ell$. Let $K\subset A_0^t$ be 
  the product of $\ker (\alpha_\ell \varphi_\ell)^t$ for all $\ell$. 
  Let $A_1$ be
  the abelian variety such that $A_1^t=A_0^t/K$. Put $\varphi_1:=
  \pr^t: A_1\to A_0$, where $\pr: A_0^t \to A_1^t=A_0^t/K$ is the natural
  projection. Then $(A_1,\varphi_1)$ has the desired property. 

  Now we choose an integer $N$ such that $N \alpha_\ell \varphi_\ell$
  is an isogeny for all $\ell$. By what we just proved, there is a unique
  member $(A_1,\varphi_1')$ such that $(A_1(\ell),\varphi_{1,
  \ell}')=(A(\ell), N  \alpha_\ell \varphi_\ell)$ for all $\ell$. Then
  $(A_1,\varphi_1'/N)$ satisfies \eqref{eq:smf.8}. 
\end{proof}

\begin{theorem}\label{smf.3}
  We have $G_x(\A_f)=G_{x,\A_f}$ if one of the following conditions
  holds:
  \begin{itemize}
  \item[(a)] $k$ is a finitely generated field.
  \item[(b)] $\fchar k=p>0$, $A_0$ is supersingular and $k$ is
    sufficiently large for $A_0$.
  \end{itemize}
\end{theorem}

We say that the ground field $k$ is sufficiently large for an abelian
variety $A$ if $\End(A)=\End(A\otimes_k k_s)$. 


\begin{proof}
Statement (a) follows from Tate's theorem on 
homomorphisms of abelian
varieties, due to Tate, Zarhin, Faltings and de Jong
(cf.~\cite[Theorem 2.1]{yu:smf}. Statement (b)
follows from the validity of Tate's theorem for supersingular abelian
varieties over a sufficiently large field $k$, 
which is well known. 
\end{proof}

\begin{prop}\label{smf.4}
  Let $\Lambda_x:=\Lambda(x)$ denote the set of isomorphism classes in
  the genus containing $x=(A_0,{\rm id})\in \wt \calA_{x}$. Assume
  that one of the conditions in Theorem~\ref{smf.3} holds. 
  Then there is a
  natural isomorphism 
  \begin{equation}
    \label{eq:smf.9} 
    \Lambda_x\simeq G_x(\Q)\backslash G_x(\A_f)/G_x(\wh \Z)
  \end{equation}
  which sends $[A_0]$ to the identity class. 
\end{prop}
\begin{proof}
  According to Definition~\ref{smf.1} (4) (iii)-(v), the genus
  containing $x$ is $G_{x,\A_f}/\Stab(x)$. Under the condition of (a)
  or (b), we have $G_{x,\A_f}=G_x(\A_f)$ and $\Stab(x)=G_x(\hat \Z)$
  by Theorem~\ref{smf.3}. It is clear that
   two members in the genus are isomorphic
  if and only if they are $G_x(\Q)$-equivalent. This proves the
  proposition. 
\end{proof}


\subsection{Genera and idealcomplexes of abelian varieties with
  additional structures}
\label{sec:idealc}

We consider a PEL-type setting. Let $(B,*)$ be finite-dimensional
semi-simple $\Q$-algebra with positive involution $*$, and $O_B\subset
B$ be a $\Z$-order stable under the involution $*$. When $O_B=\Z$,
this setting specializes to the consideration of polarized abelian
varieties.  

\begin{defn}\label{smf.5}
(1) 
\begin{itemize}
\item[(i)] 
A {\it polarized
$O_B$-abelian variety} is a triple $\ul A=(A,\lambda,\iota)$, 
where $(A,\lambda)$
is a polarized abelian variety and $\iota:O_B\to \End(A)$ is a ring
monomorphism such that $\iota(a)^t \lambda=\lambda \iota(a^*) $ for
all $a\in O_B$. 

\item[(ii)] A {\it fractional polarization} or {\it
  $\Q$-polarization}
 is an element
$\Hom(A, A^t)\otimes_\Z \Q$ such that $N \lambda$ is a polarization
for some positive integer $N$; we call a fractional polarization {\it
  integral} if $N$ can be chosen to $1$, i.e. it is a polarization.

\item[(iii)] A {\it fractionally
polarized (or\, $\Q$-polarized) $O_B$-abelian variety} is a triple
$(A,\lambda', \iota)$ such that $(A,N \lambda', \iota)$ is a polarized
$O_B$-abelian variety for some positive integer $N$.
  

\item[(iv)]
For any prime $\ell$, write $\ul A(\ell)=(A(\ell),
\lambda_\ell,\iota_\ell)$, where $\lambda_\ell: A(\ell)\to
A^t(\ell)=A(\ell)^t$ is induced morphism for $\ell$-divisible groups,
and $\iota_\ell: O_B\otimes \Z_\ell \to \End(A)\otimes \Z_\ell \to
\End(A(\ell))$ is the $\Z_\ell$-linear extension of $\iota$. 
\end{itemize}

\npr (2) Let $\ul A_1=(A_1,\lambda_1,\iota_1)$ and $\ul
  A_2=(A_2,\lambda_2,\iota_2)$ be two polarized $O_B$-abelian
  varieties over a field $k$. 
  \begin{itemize}
  \item[(i)] Denote by $\Qisog(\ul A_1,\ul A_2)$ (resp.~$\Gisom(\ul A_1,\ul
    A_2)$) the set of all
    $O_B$-linear quasi-isogenies (resp.~isomorphism) 
   $\varphi:A_1\to A_2$ such that
    $\varphi^* \lambda_1=\lambda_2$ (resp.~$\varphi^*
    \lambda_1=c \lambda_2$ for some $c\in \Q^\times$). 

  \item[(ii)] For any prime $\ell$, denote by
    $\Qisog(\ul A_1(\ell),\ul A_2(\ell))$ (resp.~$\Gisom(\ul A_1(\ell),\ul
    A_2(\ell))$) the set of all $O_B\otimes
    \Z_\ell$-linear quasi-isogenies (resp.~isomorphism) 
  $\varphi_\ell: A_1(\ell)\to
    A_2(\ell)$ such that $\varphi_\ell^*
    \lambda_{2,\ell}=\lambda_{1,\ell}$ (resp.~$\varphi_\ell^*
    \lambda_{2,\ell}=c_\ell \lambda_{1,\ell}$ for some $c_\ell\in
    \Q_\ell^\times$). 
  
  \item[(iii)] We say $\ul A_1$ and $\ul A_2$ (resp.~$\ul A_1(\ell)$ and 
  $\ul A_2(\ell)$) are {\it isogenous} if $\Qisog(\ul A_1, \ul
  A_2)\neq \emptyset$ (resp.~$\Qisog(\ul A_1(\ell), \ul
  A_2(\ell))\neq \emptyset$). 
  
  \item[(iv)] We say $\ul A_1$ and $\ul A_2$ (resp.~$\ul A_1(\ell)$ and 
  $\ul A_2(\ell)$) are {\it similar} if 
  $\Gisom(\ul A_1, \ul
  A_2)\neq \emptyset$ (resp.~$\Gisom(\ul A_1(\ell), \ul
  A_2(\ell))\neq \emptyset$).  
  \end{itemize}

\npr (3) Let $\ul x=\ul A_0=(A_0,\lambda_0,\iota_0)$ be a polarized
    $O_B$-abelian variety over $k$. 
    \begin{itemize}
    \item[(i)] Denote by $G_{\ul x}$ the automorphism group scheme
      over $\Z$ 
      associated to $\ul A_0$. It is a group scheme which represents
      the functor 
      \begin{equation}
        \label{eq:smf.10}
        R\mapsto G_{\ul x}(R)=\{g\in (\End_{O_B}(A_0)\otimes_\Z
        R)^\times \mid g^t \lambda g=\lambda\}. 
      \end{equation}
    \item[(ii)] Denote by $GU_{\ul x}$ the group scheme of similitudes over
      $\Z$ 
      associated to $\ul A_0$. It is a group scheme which represents
      the functor 
      \begin{equation}
        \label{eq:smf.11}
        R\mapsto GU_{\ul x}(R)=\{g\in (\End_{O_B}(A_0)\otimes_\Z
        R)^\times \mid \exists\, c(g)\in R^\times\ {\rm s.t.}\ g^t 
       \lambda g=\lambda c(g)\}. 
      \end{equation}
    \item[(iii)] For any prime $\ell$, write $\ul x_\ell=\ul A_0(\ell)$ and
      define the group schemes $G_{\ul x_\ell}$ and $GU_{\ul x_\ell}$
      over $\Z_\ell$ in the same way.
    \end{itemize}

\npr (4) Let $x=A_0$ be the underlying abelian variety of $\ul x=\ul A_0$.
\begin{itemize}
\item[(i)] $  \wt \calA_{\ul x}^r:=\left \{(\ul A,\varphi) \,  {\Bigg |}  \, 
\parbox{2in}{$\ul A$ is a
  $\Q$-polarized $O_B$-abelian variety over $k$ and $\varphi\in
  \Qisog(\ul A, \ul A_0)$} \right \}\subset \wt \calA_x$.

\item[(ii)] Let $G_{\ul x,\A_f}:=\prod_{\ell}' G_{\ul x_\ell}(\Q_\ell)$ and
  $GU_{\ul 
  x,\A_f}:=\prod_{\ell}' GU_{\ul x_\ell}(\Q_\ell)$ be the restricted
  products of local groups $G_{\ul x_\ell}(\Q_\ell)$ and $GU_{\ul
    x_\ell}(\Q_\ell)$ with respect to $G_{\ul x_\ell}(\Z_\ell)$ and
  $GU_{\ul x_\ell}(\Z_\ell)$, respectively. The group action of
  $G_{x,\A_f}$ on $\wt \calA_x$ induces a group action of 
  $G_{\ul x,\A_f}$ and of $GU_{\ul x,\A_f}$ on $\wt \calA_{\ul x}^r$.

\end{itemize}

\npr (5) Let $(\ul A_1,\varphi_1)$ and $(\ul A_2, \varphi_2)$ be two
  members of $\wt \calA_{\ul x}^r$.

\begin{itemize}
\item[(i)] $(\ul A_i,\varphi_i)$ ($i=1,2$) are said to be
  in the same {\it genus} if $\Isom(\ul A_1(\ell),\ul A_2(\ell))\neq
  \emptyset$ for all primes $\ell$. It is easy to see that two
  members are in the same genus (resp.~are isomorphic) 
  if and only they are $G_{\ul
  x,\A_f}$-equivalent (resp.~$G_{\ul x}(\Q)$-equivalent). The
  isomorphism class of $\ul A_1$ is denoted by $[\ul A_1]$.

\item[(ii)] $(\ul A_i,\varphi_i)$ ($i=1,2$) are said to be
  in the same {\it idealcomplex}\footnote{Our terminology borrows from the similar notion in the theory of lattices in \cite[p.~10]{ponomarev:aa1976}} if $\ul A_1(\ell)$ and $\ul
  A_2(\ell)$ are similar 
  for all primes $\ell$. It is easy to see that two
  members are in the same idealcomplex (resp.~are similar) if and
  only they are $GU_{\ul x,\A_f}$-equivalent 
  (resp.~are $GU_{\ul x}(\Q)$-equivalent). The similitude class of
  $\ul A_1$ is denoted by $[\ul A_1]_s$.

\item[(iii)] $(\ul A_i,\varphi_i)$ ($i=1,2$) are said to be
  in the same {\it similitude genus} if there is an element $\alpha\in
  GU_{\ul x}(\Q)$ such that the members 
  $\alpha \cdot (\ul A_1,\varphi_1)$ and 
  $(\ul A_2,\varphi_2)$ are in the same genus. 
  \end{itemize}

\npr (6) Let $\ul x'=(\ul A', \varphi')$  be a member of $\wt \calA_{\ul
    x}^r$.
\begin{itemize}
\item[(i)] Let 
  $\wt \Lambda_{\ul x'}\subset \wt \calA_{\ul x}^r$ be the genus 
  containing $(\ul A', \varphi')$,
  and $\Lambda_{\ul x'}:=[\wt \Lambda_{\ul x'}]$ the set of
  isomorphism classes of all members in $\wt \Lambda_{\ul x'}$. One 
  has $\Lambda_{\ul x'}=G_{\ul x}(\Q)\backslash \wt \Lambda_{\ul x'}$.

\item[(ii)] Let 
  $\wt \calI_{\ul x'}\subset \wt \calA_{\ul
    x}^r$ be the idealcomplex containing $(\ul A', \varphi')$,
  and $\calI_{\ul x'}:=[\wt \calI_{\ul x'}]_s$ the set of
  similitude classes of all members in $\wt \calI_{\ul x'}$. One has 
  $\calI_{\ul x'}=GU_{\ul x}(\Q)\backslash \wt \calI_{\ul x'}$.

\item[(iii)] Let 
  $\wt \Lambda^s_{\ul x'}\subset \wt \calA_{\ul x}^r$ be the similitude 
  genus 
  containing $(\ul A', \varphi')$,
  and $\Lambda_{\ul x'}^s:=[\wt \Lambda^s_{\ul x'}]_s$ the set of
  similitude classes of all members in $\wt \Lambda^s_{\ul x'}$. One 
  has $\Lambda^s_{\ul x'}=GU_{\ul x}(\Q)\backslash \wt 
  \Lambda^s_{\ul x'}$. 
\end{itemize}

\end{defn}

We now introduce the notion of basic abelian varieties with additional
structures. The concept is originally defined in
Kottwitz~\cite{kottwitz:isocrystals} in the content of isocrystals
with $G$-structures and the framework of Tannakian categories. 
Basic abelian varieties with additional
structures are first introduced in Rapoport and Zink
\cite{rapoport-zink} which rely on the construction of integral
models of PEL-type Shimura varieties. 
A convenient definition which does not require
the background on Shimura varieties is as follows (see \cite{yu:smf}).


\begin{defn} \label{5.6} Let $(B,*)$ and $O_B$ remain as above.

{\rm (1)}  
  Let $(V_p,\psi_p)$ be a $\Q_p$-valued non-degenerate skew-Hermitian
  $B_p$-module, where $B_p:=B\otimes_{\Q} \Q_p$. A polarized $O_B$-abelian
  variety $\ul A=(A,\lambda,\iota)$ over an \ac field
  $k$ of \ch $p$
   is said
  to be {\it related to $(V_p,\psi_p)$} if there is a 
  $B_p\otimes_{\Qp} L$-linear isomorphism 
  $\alpha:M(\ul A)\otimes_W L\simeq
  (V_p,\psi_p)\otimes_{\Qp} L$ which preserves the pairings up to a scalar
  in $L^\times$, where $W$ is the ring of Witt vectors over $k$, $L$
  the fraction field of $W$, and  
  $M(\ul A)$ is the
  covariant \dieu module with additional structures associated to $\ul A$. 

  Let $G':=GU_{B_p}(V_p,\psi_p)$ be the algebraic
  group over $\Q_p$ of $B_p$-linear similitudes on $(V_p, \psi_p)$. 
  By transporting the structure of the Frobenius map on $V_p\otimes L$, we
  obtain an element $b\in G'(L)$ such that 
  \begin{equation}\label{eq:smf.115}
\alpha:M(\ul A)\otimes L \simeq (V_p\otimes L, \psi_p, b({\rm id}\otimes
  \sigma))    
  \end{equation}
is an isomorphism of isocrystals with additional structures. 
  The decomposition of $V_p\otimes L$ into
  isoclinic components induces a $\Q$-graded structure, and thus
  defines a (slope) homomorphism $\nu_{b}:\D\to G'$ over $L$, 
  where $\D$ is the pro-torus over
  $\Q_p$ with character group $\Q$.    

   {\rm (2)} A polarized $O_B$-abelian variety $\ul A$ over an \ac
  field $k$ of \ch $p$ is said to  {\it
  basic with respect to $(V_p,\psi_p)$} if 
\begin{itemize}
\item [(a)] $\ul A$ is related to $(V_p,\psi_p)$, and
\item [(b)] the slope homomorphism $\nu_b$ is central.
\end{itemize} 
Note that this is independent of the choice of $\alpha$ and $b$.
We call $\ul A$ {\it basic} if it is basic
  with respect to $(V_p,\psi_p)$ for some skew-Hermitian space
  $(V_p,\psi_p)$.


{\rm (3)} An polarized $O_B$-abelian variety $\ul A$ over an arbitrary
  field $k$ of \ch $p$ is said to be  {\it basic} if the base change 
  $\ul A\otimes_k \bar k$ over its algebraic closure $\bar k$ is basic
\end{defn}

One can choose a suitable isomorphism $\alpha$ in \eqref{eq:smf.115} such
that the slope homomorphism $\nu_b:\bfD\to G'$ is defined over
$\Q_{p^s}$ for some positive integer $s$ 
\cite[Section 4.3]{kottwitz:isocrystals}. However, 
this fact is not needed in Definition~\ref{5.6}. 
It is clear that the notion of basic abelian varieties with additional
structures is preserved under isogeny. By \cite[Theorem
1.1]{yu:avf_basic}, any basic abelian varieties with additional
structures over an arbitrary \ac field of \ch $p$ is isogenous to
another one which is defined a finite field. It follows that the
notion of basic abelian varieties with additional structures does not
depend on the ground field over which they are defined. 


\begin{theorem}\label{smf.6}
  Assume one of the following conditions
  holds:
  \begin{itemize}
  \item[(a)] $k$ is a finitely generated field.
  \item[(b)] $\fchar k=p>0$, $\ul A_0$ is basic  and $k$ is
    sufficiently large for $A_0$.
  \end{itemize}
We have $G_{\ul x}(\A_f)=G_{\ul x,\A_f}$ and 
  $GU_{\ul x}(\A_f)=GU_{\ul x,\A_f}$.
\end{theorem}

\begin{proof}
  Statement (a) follows directly from
  Theorem~\ref{smf.3}. Statement (b) is proved in 
  \cite[Theorem 3.12 and Proposition 4.5]{yu:smf}. 
\end{proof}

Let $G$ be a linear algebraic group over $\Q$ and $U\subset G(\A_f)$
an open compact subgroup. Denote by ${\rm DS}(G,U)$ for the double
coset space $G(\Q)\backslash G(\A_f)/U$. This is a finite set by a
finiteness result of Borel and Harish-Chandra \cite{Borel-HC:1962}. 
Assume that any arithmetic
subgroup of $G(\Q)$ is finite. Let $c_1,c_2,\dots, c_h$ be a complete
coset representatives for ${\rm DS}(G,U)$ and put $\Gamma_i:=G(\Q)\cap
c_i U c_i^{-1}$ for $i=1,\dots, h$. Define the mass of $(G,U)$ by 
\begin{equation}
  \label{eq:smf.12}
  \Mass(G,U):=\sum_{i=1}^h |\Gamma_i|^{-1}.
\end{equation}

For any finite set $S$ consisting of objects with finite automorphism
groups, define the mass of $S$ by
\begin{equation}
  \label{eq:smf.13}
  \Mass(S):=\sum_{s\in S} |\Aut(s)|^{-1}.
\end{equation}

\begin{theorem}\label{smf.7}
  Let $\Lambda_{\ul x}$ (resp.~$\calI_{\ul x}$; resp.~$\Lambda_{\ul x}^s$) 
  the set of isomorphism classes (resp. similitude classes) in
  the genus (resp.~the idealcomplex; resp.~the similitude genus) 
  containing the base member $\ul x=(\ul A_0,{\rm id})\in \wt
  \calA_{\ul x}^r$. Assume 
  that one of conditions in Theorem~\ref{smf.6} holds. \\
{\rm (1)} There are natural isomorphisms sending the base class to the
  identity class: 
  \begin{align}
    \label{eq:smf.14}
    \Lambda_{\ul x}  & \simeq {\rm DS}(G_{\ul x},G_{\ul x}(\wh \Z)), \\  
    \label{eq:smf.15}
    \calI_{\ul x}  & \simeq {\rm DS}(GU_{\ul x},GU_{\ul x}(\wh \Z)). \\
    \label{eq:smf.16}
    \Lambda_{\ul x}^s  &\simeq GU_{\ul x}(\Q)\backslash GU_{\ul x}(\Q) G_{\ul
    x}(\A_f)/G_{\ul x}(\wh \Z).  
  \end{align}
{\rm (2)} We have 
\begin{align}
  \label{eq:smf.17}
  \Mass(\Lambda_{\ul x}) & =\Mass(G_{\ul x},G_{\ul x}(\wh \Z)),\\
   \label{eq:smf.18}
  \Mass(\calI_{\ul x}) & =\Mass(GU_{\ul x},GU_{\ul x}(\wh \Z)),\\
   \label{eq:smf.19}
  \Mass(\Lambda^s_{\ul x}) & =\sum_{i=1}^h |\Gamma_i|^{-1}, 
\end{align}
where $c_1,\dots, c_h$ are coset representatives for 
$GU_{\ul x}(\Q)\backslash GU_{\ul x}(\Q) G_{\ul
    x}(\A_f)/U$, $U=G_{\ul x}(\wh \Z)$, and $\Gamma_i:=GU_{\ul
    x}(\Q)\cap c_i U c_i^{-1}$.
\end{theorem}

\begin{proof}
Statements \eqref{eq:smf.14} and \eqref{eq:smf.17} are proved in 
\cite[Theorems 2.2 and 4.6]{yu:smf}. The same proofs also prove
\eqref{eq:smf.15} and \eqref{eq:smf.18}, because we have 
$GU_{\ul x,\A_f}=GU_{\ul x}(\A_f)$ for cases (a) and (b) in 
Theorem~\ref{smf.6}. By definition, 
\[\wt \Lambda^s_{\ul x}=GU_{\ul x}(\Q) \cdot \wt \Lambda _{\ul
  x}=GU_{\ul x}(\Q) G_{\ul 
    x}(\A_f)/G_{\ul x}(\wh \Z),\] and we get
  \eqref{eq:smf.16}. Formula \eqref{eq:smf.19} follows from
  \eqref{eq:smf.16} 
  and \eqref{eq:smf.18} because $\Lambda^s_{\ul x}$ is a subset of
  $\calI_{\ul x}$.  
\end{proof}

\subsection{Principally polarized abelian varieties over finite
  fields}
\label{sec:smf.4}

Let $\ul x_0=(A_0,\lambda_0)$ be a principally polarized abelian
variety over 
$\Fq$. Then the group schemes $GU_{\ul x_0}$ and $G_{\ul x_0}$ 
represents the functor 
\begin{equation}
  \label{eq:smf.20}
  \begin{split}
  &R \mapsto \{x\in (\End(A_0)\otimes R)^\times |\, x x'\in R^\times
  \},\quad \text{and}    \\
  &R \mapsto \{x\in (\End(A_0)\otimes R)^\times |\, x x'=1
  \},  
  \end{split}
\end{equation}
respectively, where $'$ is the Rosati involution on $\End(A_0)$
induced by $\lambda_0$. We denote by $\calI_{\ul x_0}$ the set of
similitude classes of the idealcomplex containing $(A_0,\lambda_0)$, 
and by $\Lambda_{\ul x_0}$ the set of isomorphism classes of the genus
containing $(A_0,\lambda_0)$. By Theorem~\ref{smf.7}, we have 
natural bijections  

\begin{equation}\label{eq:smf.21}
  \begin{split}
  \Lambda_{\ul x_0}&\simeq G_{\ul x_0}(\Q)\backslash G_{\ul
    x_0}(\A_f)/G_{\ul x_0}(\wh \Z), \quad     \\
  \calI_{\ul x_0}&\simeq GU_{\ul x_0}(\Q)\backslash GU_{\ul
    x_0}(\A_f)/GU_{\ul x_0}(\wh \Z).
  \end{split}
\end{equation}
 
We call an abelian variety $A$ over a finite field \emph{principal} if
the endomorphism algebra $\End^0(A)$ of $A$ is commutative and the
endomorphism ring $\End(A)$ is the maximal order.\footnote{This notion is more general than that defined in \cite[Section 5]{waterhouse:thesis}; the abelian variety $A$ does not need to be simple here.} A genus of abelian varieties which contains
 a principal abelian variety is called a principal genus.

Suppose that $A_0$ is principal. Then $\End^0(A_0)$ is a CM algebra
$K=\prod_{i=1}^r K_i$ with each $K_i$ a CM field, and $'$ is the canonical
involution on $K$. Let $K^+=\prod_{i=1}^r K^+_i$ be the maximal totally
real subalgebra, which is the fixed subalgebra of $'$. We denote 
the maximal orders of $K$ and $K^+$ by $O_K$ and $O_{K^+}$, 
respectively. Let $T=T^{K,\Q}$ be the algebraic ``torus'' over
$\Z$ defined by $T(R)=\{x\in (O_K\otimes_{\Z} R)^\times |
N_{K/K^+}(x)=x x'\in 
R^\times\}$ for any commutative ring $R$. Let $T^{K,1}$ be the 
$K/K^+$-norm one
subtorus of $T$. 
It follows from the definition that $G_{\ul x_0}=T^{K,1}$ and $GU_{\ul x_0}=T^{K,\Q}$.
Recall the class number $h_T$ of $T$ is the cardinality of 
$T(\Q)\backslash T(\A_f)/U_T$, where $U_T$ is the maximal open compact
subgroup of 
$T(\A_f)$.  
Since $T(\wh \Z)$ and $T^{K,1}(\wh \Z)$ are maximal
open compact subgroups in the adelic groups, by \eqref{eq:smf.21}, we
get
\begin{equation}
  \label{eq:smf.22}
  |\Lambda_{\ul x_0}|=h_{T^{K,1}}\quad  
  \text{and }\quad |\calI_{\ul x_0}|=h_{T}.
\end{equation}

\begin{prop}\label{smf.9}
  Let $K=\prod_{i=1}^r K_i$ be a CM algebra with maximal totally real
  subalgebra $K^+=\prod_{i=1}^r K_i^+$. 
  
{\rm (1)} We have
\begin{equation}
  \label{eq:smf.23}
  h_{T^{K,1}}=\frac{h_K}{h_{K^+}}\cdot \frac{1}{Q\cdot 2^{t-r}}, 
\end{equation}
where $h_K:=|\Pic(O_K)|$ (resp.~$h_{K^+}$) is the class number of $K$
(resp.~$K^+$), $Q:=[O^\times_K:\mu_K \cdot O^\times_{K^+}]$ is the
Hasse unit index, and $t$ is the number of finite places of $K^+$
ramified in $K$.

{\rm (2)} We have 
\begin{equation}
  \label{eq:smf.24}
  h_{T^{K,\Q}}=\frac{h_K}{h_{K^+}}\cdot \frac{1}{Q\cdot 2^{t-r}} \cdot \frac{\prod_{p\in S_{K/K^+}} [\Z_p^\times:N(T^{K,\Q}(\Z_p))]}{[\A^\times:N(T^{K,\Q}(\A)) \cdot \Q^\times]}, 
\end{equation}   
where $S_{K/K^+}$ is the set of primes $p$ such that there exists a place $v|p$ of $K^+$ which is ramified in $K$.
\end{prop}  

The formula \eqref{eq:smf.23} is proved when $K$ is a CM field; see
\cite[(16), p.~375 ]{shyr:michigan1977}, When $K$ is a CM algebra, the
formula then follows from $h_{T^{K,1}}=\prod_{i=1}^r h_{T^{K_i,1}}$
and loc.~cit. For the formula \eqref{eq:smf.24}, see  \cite[Theorem 1.1]{guo-sheu-yu:CM}. 


By Proposition~\ref{smf.9} and \eqref{eq:smf.22},
we obtain the following result.

\begin{thm}\label{smf.10}
Let $\ul x_0=(A_0,\lambda_0)$ be a principally polarized principal 
abelian variety over $\Fq$ with endomorphism algebra
$\End^0(A_0)=K=\prod_{i=1}^r K_i$. Then the genus $\Lambda_{\ul x_0}$
of $\ul x_0$ has  
\[ \frac{h_K}{h_{K^+}}\cdot \frac{1}{Q\cdot 2^{t-r}} \]   
isomorphism classes. The idealcomplex $\calI_{\ul x_0}$
of $\ul x_0$ has 
\[ \frac{h_K}{h_{K^+}}\cdot \frac{1}{Q\cdot 2^{t-r}} \cdot \frac{\prod_{p\in S_{K/K^+}} [\Z_p^\times:N(T^{K,\Q}(\Z_p))]}{[\A^\times:N(T^{K,\Q}(\A)) \cdot \Q^\times]} \]
similitude classes. 
\end{thm}

\section*{Acknowledgments}
J.~Xue is partially supported by the 1000-plan program for young
talents and 
Natural Science Foundation grant \#11601395 of PRC. 
C.-F. Yu is partially supported by the MoST grants 
104-2115-M-001-001MY3 and 107-2115-M-001-001-MY2.

\bibliographystyle{hplain}
\bibliography{TeXBiB}
\end{document}